\documentclass[11pt,twoside, a4paper]{amsart}

\usepackage{amsthm,amsmath,natbib,amssymb,tikz}
\usetikzlibrary{calc,fadings,decorations.pathreplacing,intersections,arrows}
\usepackage[english]{babel}
\usepackage[T1]{fontenc}
\usepackage[colorlinks,citecolor=blue,urlcolor=blue]{hyperref}

\DeclareMathOperator*{\argmin}{arg\,min}

\begin{document}

\newtheorem{theorem}{Theorem}
\newtheorem{lemma}{Lemma}
\newtheorem{corollary}{Corollary}
\newtheorem{proposition}{Proposition}
\theoremstyle{definition}
\newtheorem{dfn}{Definition}
\newtheorem{ex}{Example}
\renewcommand{\phi}{\varphi}
\renewcommand{\epsilon}{\varepsilon}

\title[Degrees of freedom for nonlinear least squares]{Degrees of freedom for nonlinear least squares estimation}
\author[N. R. Hansen]{Niels Richard Hansen} 
\address{Department of Mathematical Sciences,
University of Copenhagen,
Universitetsparken 5,
2100 Copenhagen \O,
Denmark}
\email[Corresponding author]{Niels.R.Hansen@math.ku.dk}

\author[A. Sokol]{Alexander Sokol}
\email{alexander@math.ku.dk}



\begin{abstract}
  We give a general result on the effective degrees of freedom for
  nonlinear least squares estimation. It relates the degrees of
  freedom to the divergence of the estimator. We show that in a 
  general framework, the divergence of the least squares estimator 
  is a well defined but potentially negatively biased estimate of the 
  degrees of freedom,  and we give an exact representation of the
  bias. This implies that if we use the divergence as a plug-in
  estimate of the degrees of freedom in Stein's unbiased risk estimate
  (SURE), we generally underestimate the true risk. Our result
  applies, for instance, to model searching problems, yielding a finite sample
  characterization of how much the search contributes to the
  degrees of freedom.  Motivated by the problem of fitting
  ODE models in systems biology, the general results are illustrated by 
  the estimation of systems of linear ODEs. In this example  
  the divergence turns out to be a useful estimate of degrees of
  freedom for $\ell_1$-constrained models.
\end{abstract}

\maketitle

\noindent

\section{Introduction}

The concept of effective degrees of freedom for least squares
estimation in a mean value model is a classical and well studied concept,
which is intimately related to and useful for model assessment and selection, see
e.g. \cite{Hastie:1990gam}, \cite{Ye:1998}, \cite{Efron:2004}. The
more recent interest in the concept has focused on the computation and estimation of
degrees of freedom for non-smoothly penalized or constrained
mean value models. The case of $\ell_1$-penalized least squares estimation in
\emph{linear} models has recieved considerable attention, and
\cite{Tibshirani:2012} provide the most complete results. Convexity has been
pivotal for these recent theoretical developments. In the constrained formulation the mean value model itself must
be convex, and in the penalized formulation the results rely on duality theory
from convex optimization. As we argue below, there are important
applications in systems biology where the mean value models are
inherently nonlinear and non-convex. Realistic models are complex and multivariate,
and the amount of data is limited, so asymptotic arguments are
difficult to justify. Thus for the development of 
appropriate small sample methods for model assessment, a detailed understanding of the effective
degrees of freedom is very useful. We give results on the effective degrees of freedom for
the completely general case where the mean value model is a closed set
and the mean is estimated by least squares. We show
that the classical estimator of the degrees of freedom -- the divergence of the mean
value estimator -- is always well defined but generally biased. The
bias arise from the non-convex geometry of the mean value model, and we
show how the non-convexity is encoded into a Radon measure, and how this
measure gives an explicit formula for the bias.  

Our main motivation for considering non-convex mean value models is 
for estimation of continuous time dynamical models 
from experimental data as is encountered in systems biology, see e.g. \cite{Wilkinson:2006},
\cite{Montefusco:2011} or \cite{Oates:2012}. Multivariate ODE models
constitute an important model class in this area. Despite the many existing
approaches in the literature, data driven estimation and selection of
a multivariate continuous time dynamical model remains a non-trivial
problem. The challenges include the development of methods that scale
well with the dimension of the model, as well as feasible methods to honestly assess the statistical
uncertainty and to avoid overfitting. Through several 
approximations within the continuous time dynamical models, \cite{Oates:2012} managed
to recast aspects of the estimation problem (estimation of the network) in a unifying 
framework relying on the linear model. Though this allowed
for the use of a range of regularization or model selection methods
for the linear model, the conclusion was that ``biological network
inference remains profoundly challenging''. In addition, they observed that 
experimental designs with uneven sampling intervals represented particular
difficulties. We believe that one of the difficulties lies in the
approximations within the continuous time models, which become
particularly pronounced for large sampling intervals. To overcome this
problem and avoid the approximations, we need to consider estimation of
the continuous time models directly, which inevitably leads to
nonlinear mean value models.

We suggest that the challenges in systems biology outlined above may be approached by
non-smooth regularization methods for estimation of parameters in
multivariate ODE models. For this reason we consider 
$\ell_1$-constrained nonlinear least squares
estimation as a main example in the present paper. Our theoretical results do, however,
apply to the general class of least squares
estimators that are given by a possibly non-convex constraint on the
mean value. 
Notably, they apply to estimators obtained by model searching.

In the remaining part of this introduction we describe the general
setup in more details, and we outline the contributions of the
paper. The objective is the assessment of the risk of 
nonlinear least squares estimators using Stein's unbiased risk estimate (SURE), 
as treated in e.g. \cite{Efron:2004}. SURE provides a
non-asymptotic and unbiased estimate of the risk for general mean value
estimators, if we can estimate the effective degrees of
freedom unbiasedly. This was considered in \cite{Meyer:2000} and \cite{Kato:2009} for the
projection onto a closed convex set, and in \cite{Efron:2004b},
\cite{Zou:2007} and \cite{Tibshirani:2012} for $\ell_1$-penalized 
least squares estimation. Unbiased estimation of the effective degrees
of freedom relies on Stein's lemma, which does not
hold in general -- as we will show -- for nonlinear least squares
estimation. Our main result, Theorem \ref{thm:main}, 
is a generalization of Stein's lemma.

We consider the setup where $\mathbf{Y} \sim
\mathcal{N}(\boldsymbol{\xi}, \sigma^2 I_n)$, $\boldsymbol{\xi} \in \mathbb{R}^n$ and $\sigma^2
> 0$. The objective is to estimate $\boldsymbol{\xi}$. With $K \subseteq \mathbb{R}^n$ a
nonempty closed set, and
\begin{equation} \label{eq:metricproj}
\mathrm{pr}(\mathbf{y}) \in \argmin_{\mathbf{x} \in K} ||\mathbf{y} - \mathbf{x}||^2_2
\end{equation}
denoting a point that minimizes the Euclidean distance from $\mathbf{y} \in
\mathbb{R}^n$ to $K$, we estimate $\boldsymbol{\xi}$ by $\mathrm{pr}(\mathbf{Y})$. We do not require that
$\boldsymbol{\xi}$ belongs to $K$. The map $\mathrm{pr} :
\mathbb{R}^n \to \mathbb{R}^n$ defined by (\ref{eq:metricproj}) is known as the metric projection onto
$K$. Though it may not be uniquely defined everywhere, it is, in fact,
Lebesgue almost everywhere unique. For the purpose of this
introduction we assume that a (Borel measurable) selection has been made on
the Lebesgue null set where the metric projection is not unique. 

We may think of $K$ as the image of a
parametrization, that is, for a map $\zeta : \mathbb{R}^p \to
\mathbb{R}^n$ and a closed set $\Theta \subseteq \mathbb{R}^p$ it
holds that
\begin{equation} \label{eq:par}
K = \zeta(\Theta).
\end{equation}
The setup thus
includes most linear and nonlinear regression models, and the
estimator $\mathrm{pr}(\mathbf{Y})$ is the least squares estimator. Moreover,
by taking parameter sets of the form 
$$\Theta = \{ \beta \in \mathbb{R}^p \mid J(\beta) \leq s\}$$
for $s \geq 0$ and some function $J : \mathbb{R}^p \to [0, \infty)$,
the setup includes many regularization methods in their constrained formulation, see Figure \ref{fig:imageball}. If $\zeta$ is
continuous and $\Theta$ is bounded in addition to being closed, then
$K$ is compact and thus automatically closed. The assumption that $K$ is closed is
the only regularity assumption we require for the general results to
hold. Note, in particular, that $K$ is not assumed convex as in \cite{Meyer:2000}
and \cite{Kato:2009}. For convex $K$, the metric projection is
Lipschitz, which implies that Stein's lemma holds. The novelty of our results is that they
apply without a convexity assumption on $K$.  

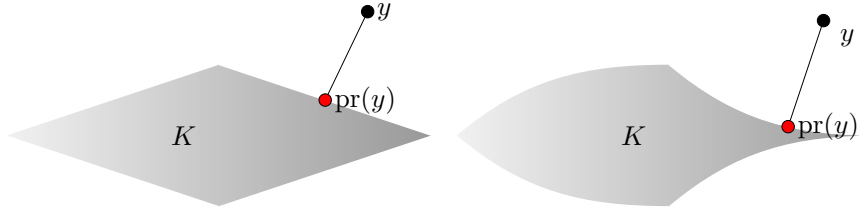
\begin{figure}
\begin{center}
\begin{tabular}{cc}
\begin{tikzpicture} [scale = 0.5]

\shade[left color=gray!10,right color=gray!80] 
  (0,0) -- (6,-2) -- (12,0) -- (6,2) -- cycle;

\draw[color=white] 
  (12,0) -- 
  coordinate[pos=.5] (cux1) (6,2);

\draw (cux1){}+(1.2,2.5) -- (cux1);
\draw[fill=black] (cux1){}+(1.2,2.5) circle (5pt);
\draw[fill=red] (cux1) circle (5pt);
\node at (5, 0) {$K$};
\node[anchor=west] at (cux1) {$\mathrm{pr}(y)$};
\node[anchor=west] at ($(cux1)+(1.2,2.5)$) {$y$};

\end{tikzpicture} &

\begin{tikzpicture} [scale = 0.5]

\shade[left color=gray!10,right color=gray!80] 
  (0,0) to[out=-40,in=-180] (6,-2) to[out=40,in=-180]  (12,0)
  to[out=-180,in=-40] (6,2)  
  to[out=180,in=40] (0,0);

\draw[color=white] 
  (12,0) to[out=-180,in=-40] 
  coordinate[pos=.4] (cux1) (6,2);

\draw (cux1){}+(1,3) -- (cux1);
\draw[fill=black] (cux1){}+(1,3) circle (5pt);
\draw[fill=red] (cux1) circle (5pt);
\node at (5, 0) {$K$};
\node[anchor=west] at (cux1) {$\mathrm{pr}(y)$};
\node[anchor=west] at ($(cux1)+(1.2,2.5)$) {$y$};

\end{tikzpicture}
\end{tabular}
\end{center}
\caption{Illustration of a metric projection in $\mathbb{R}^n$ onto the image of an
  $\ell_1$-ball using a linear (left) or a nonlinear (right)
  parametrization. \label{fig:imageball}}
\end{figure}

With $$\mathrm{Risk} = E||\boldsymbol{\xi} - \mathrm{pr}(\mathbf{Y})||^2_2$$
denoting the risk of the estimator, it is well known that 
\begin{equation} \label{eq:riskid}
\mathrm{Risk} = E ||\mathbf{Y} - \mathrm{pr} (\mathbf{Y})||^2_2 - n \sigma^2 + 2
\sigma^2 \mathrm{df}
\end{equation}
where 
\begin{equation} \label{eq:dfdf}
\mathrm{df} = \frac{1}{\sigma^2}\sum_{i=1}^n \mathrm{cov}(Y_i,
\mathrm{pr}_i(\mathbf{Y})).
\end{equation}
See e.g. \cite{Tibshirani:2012}, \cite{Efron:2004} and \cite{Ye:1998}.

It turns out that the metric projection is Lebesgue almost
everywhere differentiable, see Section \ref{sec:df}, and we can
therefore introduce the Stein degrees of freedom as
$$\mathrm{df}_S = E(\nabla \cdot \mathrm{pr} (\mathbf{Y}))$$
with  $\nabla \cdot \mathrm{pr} = \sum_{i=1}^n \partial_i \mathrm{pr}_i$ denoting the divergence
of $\mathrm{pr}$. As mentioned above, if $\mathrm{pr}$
is almost differentiable, Lemma 2 (Stein's lemma) in \cite{Stein:1981}
implies that
$$\mathrm{df} = \mathrm{df}_S.$$
However, differentiability Lebesgue almost everywhere does not imply almost
differentiability, and Theorem \ref{thm:main} in
Section \ref{sec:df} gives that in general
\begin{equation} \label{eq:dfdiff}
\mathrm{df} - \mathrm{df}_S \geq 0.
\end{equation}

Theorem \ref{thm:main} also gives a characterization of $\mathrm{df} -
\mathrm{df}_S$, whose size is
closely related to the distance from $\boldsymbol{\xi}$ to points where
the metric projection is non-differentiable, and the ``magnitude'' of
the non-differentiability -- see also the discussion in Section
\ref{sec:dis}. This ``magnitude'' is in turn related to the
non-convexity of $K$, and our result is to the best of our
knowledge the first result that characterizes how
non-convexity affects the degrees of freedom, and hence the risk of 
the least squares estimator. The non-convexity of $K$ is basically
unavoidable when we consider parametrized models with a nonlinear 
parametrization $\zeta$, and it is also pivotal for dealing with
model search problems. A typical model search problem falls
within our setup by taking $K$ to be a finite union of closed sets (the
union of the different models). The
prime example is best subset selection in linear regression, which
corresponds
to $K$ being a union of subspaces. We give a more detailed treatment of 
a special case of best subset selection in Example \ref{ex:bestsubset}
and make some remarks about the general case after this example. 

It follows from (\ref{eq:riskid}) and (\ref{eq:dfdiff}) that the risk estimate 
\begin{equation} \label{eq:riskest}
\widehat{\mathrm{Risk}} = ||\mathbf{Y} - \mathrm{pr} (\mathbf{Y})||^2_2 - n \sigma^2 + 2 \sigma^2 \nabla \cdot \mathrm{pr} (\mathbf{Y})
\end{equation}
is negatively biased in general -- systematically underestimating the true
risk. Whether we can estimate or bound this bias is still an open
problem, but our characterization of $\mathrm{df} -
\mathrm{df}_S$ in Theorem \ref{thm:main} provides a way to attack this
problem. In Section \ref{sec:ode} we present the results of using 
(\ref{eq:riskest}) in the context of $\ell_1$-constrained estimation and model searching
for dynamical systems modeled using linear ODEs. To compute
$\widehat{\mathrm{Risk}}$ we need formulas for the computation of the
divergence $\nabla \cdot \mathrm{pr}$, and we give two such results 
in Section \ref{sec:nonlinear} when $K$ is given by
(\ref{eq:par}) -- with some additional regularity assumptions on the
parametrization $\zeta$.

\section{Degrees of freedom for the metric projection} \label{sec:df}

In this section we present the main general results on differentiability of
the metric projection, and how the divergence is related to the
degrees of freedom. This gives a characterization of the bias
of $\nabla \cdot \mathrm{pr}$ as an estimate of $\mathrm{df}$ in cases where the metric projection does not satisfy a
sufficiently strong differentiability condition. The proofs are given
in Section \ref{sec:proofs}.

\begin{dfn} \label{dfn:dif} With $D \subseteq \mathbb{R}^n$ we say
  that a function $f : D \to  \mathbb{R}^n$ is differentiable in $y
  \in D$ in the extended sense if there is a neighborhood $N$ of
  $y$ such that $D^c \cap N$ is a Lebesgue null set and 
$$f(\mathbf{x}) = f(\mathbf{y}) + A(\mathbf{x} - \mathbf{y}) + o(||\mathbf{x} - \mathbf{y}||_2)$$
for $\mathbf{x} \in D \cap N$ and a matrix $A$.
\end{dfn}

If $f$ is differentiable in $\mathbf{y}$ in the extended sense the matrix $A$,
depending on $\mathbf{y}$, is necessarily unique by denseness of $D \cap N$ in $N$. We
define the partial derivatives -- and thus the
divergence -- of $f$ in $\mathbf{y}$ in terms of $A$ by
$$\partial_j f_i(\mathbf{y}) = A_{ij}$$
for $i,j = 1, \ldots, n$. Note that the partial derivatives of $f$ in
$\mathbf{y}$ need not exist in the classical sense if $f$ is differentiable in
$\mathbf{y}$ in the extended sense, but if they do, they coincide with $A_{ij}$.   

\begin{theorem} \label{thm:prdif} There exists a Borel measurable choice of the
  metric projection as a map $\mathrm{pr} :
  \mathbb{R}^n \to \mathbb{R}^n$ with the property that 
$$\mathrm{pr}(\mathbf{y}) \in \argmin_{\mathbf{x} \in K} ||\mathbf{y} - \mathbf{x}||^2_2$$
for all $\mathbf{y} \in \mathbb{R}^n$. Moreover, $\mathrm{pr}(\mathbf{y})$, is
  uniquely defined and differentiable in the extended sense for
  Lebesgue almost all $\mathbf{y}$ with $\partial_i \mathrm{pr}_i(\mathbf{y}) \geq 0$ for $i = 1,
  \ldots, n$.
\end{theorem}

As a consequence of Theorem \ref{thm:prdif}, $\mathrm{pr}(\mathbf{Y})$ is
uniquely defined with probability 1, and it follows from the
triangle inequality that 
$$||\mathrm{pr}(\mathbf{Y})||_2 \leq ||\mathrm{pr}(\mathbf{0})||_2 + 2
||\mathbf{Y}||_2.$$
This shows, in particular, that $\mathrm{pr}_i(\mathbf{Y})$ has finite second moment. Moreover, 
Theorem \ref{thm:prdif} gives that the divergence $\nabla \cdot
\mathrm{pr} (\mathbf{Y})$ is well defined and positive with probability 1.
These considerations ensure that the following definition is
meaningful. 

\begin{dfn} The degrees of freedom for the metric projection as an
  estimator of $\boldsymbol{\xi}$ is defined as 
\begin{equation} \label{eq:covdef}
\mathrm{df} = \frac{1}{\sigma^2} \sum_{i=1}^{n}
\mathrm{cov}(Y_i, \mathrm{pr}_i(\mathbf{Y})),
\end{equation}
and the Stein degrees of freedom is defined as 
\begin{equation} \label{eq:steindef}
\mathrm{df}_S = E( \nabla \cdot \mathrm{pr} (\mathbf{Y})).
\end{equation}
\end{dfn}

Our next result gives the general relation between $\mathrm{df}$ and
$\mathrm{df}_S$. To this end, let 
$$\psi(\mathbf{y}; \boldsymbol{\xi}, \sigma^2) =  \frac{1}{(2 \pi \sigma^2)^{n/2}} e^{-\frac{||\mathbf{y} -
    \boldsymbol{\xi}||_2^2}{2 \sigma^2}}$$
denote the density for the distribution of $\mathbf{Y}$ -- the multivariate
normal distribution with mean vector $\boldsymbol{\xi}$ and covariance matrix $\sigma^2
I_n$.  

\begin{theorem} \label{thm:main} There exists a Radon measure $\nu$,
  singular w.r.t. the Lebesgue measure, such that 
\begin{equation} \label{eq:main}
\mathrm{df} = \mathrm{df}_S + \int_{\mathbb{R}^n} \psi(\mathbf{y}; \boldsymbol{\xi}, \sigma^2)
\nu(\mathrm{d} \mathbf{y}).
\end{equation}
\end{theorem}

The complete proof is given in Section \ref{sec:proofs}, but let us
explain the main ideas. Introducing the convex function 
\begin{equation} \label{eq:rhodf}
\rho(\mathbf{y}) = \sup_{\mathbf{x} \in K} \{ \mathbf{y}^T \mathbf{x} - ||\mathbf{x}||^2/2\},
\end{equation}
the metric projection is a subgradient of $\rho$. The proof of Theorem
\ref{thm:main} amounts to a computation of the second order
distributional derivative of $\rho$. Convexity of $\rho$
implies that the second order distributional derivatives in the coordinate
directions are represented by positive measures, whence the partial distributional derivative of
$\mathrm{pr}_i$ in the $i$'th direction is represented by a positive
measure. Partial integration based on the definition (\ref{eq:covdef}) 
gives a representation of $\mathrm{df}$ in terms of these
partial distributional derivatives. Furthermore, the $i$'th 
partial distributional derivative of $\mathrm{pr}_i$ has, as a measure, Lebesgue decomposition 
$$\partial_i \mathrm{pr}_i \cdot m_n + \nu_i,$$
where $m_n$ denotes the Lebesgue measure on $\mathbb{R}^n$ and $\nu_i
\perp m_n$. The measure $\nu$ that appears in Theorem \ref{thm:main} is given as $\nu = \sum_{i=1}^n
\nu_i$. Note that $\nu$ depends only on the closed set $K$, and is, in particular, independent
of $\boldsymbol{\xi}$ and $\sigma^2$. 

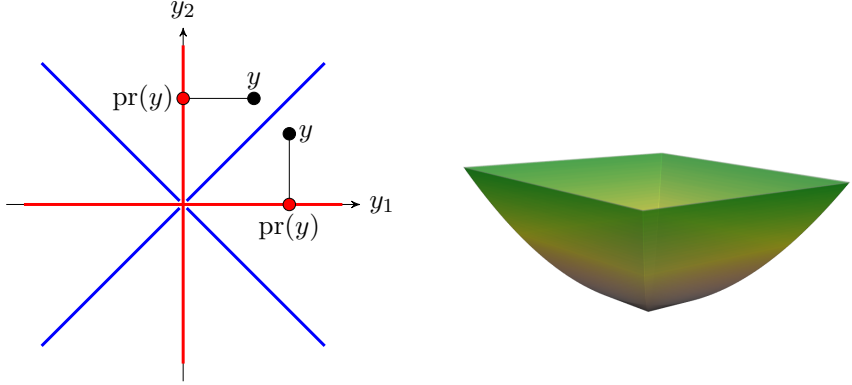
\begin{figure}
\begin{center}
\begin{tabular}{cc}
\begin{tikzpicture}[scale = 0.5]

\draw[->,  >=stealth'] (-5,0) -- (5,0) node[right] {$y_1$}; 
\draw[->,  >=stealth'] (0,-5) -- (0,5) node[above] {$y_2$}; 

\draw[very thick, color=red] (-4.5,0) -- (4.5,0); 
\draw[very thick, color=red] (0,-4.5) -- (0,4.5); 

\draw[very thick, color=blue] (-4,-4) -- (-0.1,-0.1); 
\draw[very thick, color=blue] (0.1,0.1) -- (4,4); 
\draw[very thick, color=blue] (4,-4) -- (0.1,-0.1); 
\draw[very thick, color=blue] (-0.1,0.1) -- (-4,4);

\draw (3, 2) -- (3, 0);
\draw[fill=black] (3, 2) circle (5pt);
\draw[fill=red] (3,0) circle (5pt);
\node[anchor=north] at (3, 0) {$\mathrm{pr}(y)$};
\node[anchor=west] at (3,2) {$y$};

\draw (2, 3) -- (0, 3);
\draw[fill=black] (2, 3) circle (5pt);
\draw[fill=red] (0, 3) circle (5pt);
\node[anchor=east] at (0, 3) {$\mathrm{pr}(y)$};
\node[anchor=south] at (2, 3) {$y$};

\end{tikzpicture} & 
\includegraphics[width=0.5\textwidth]{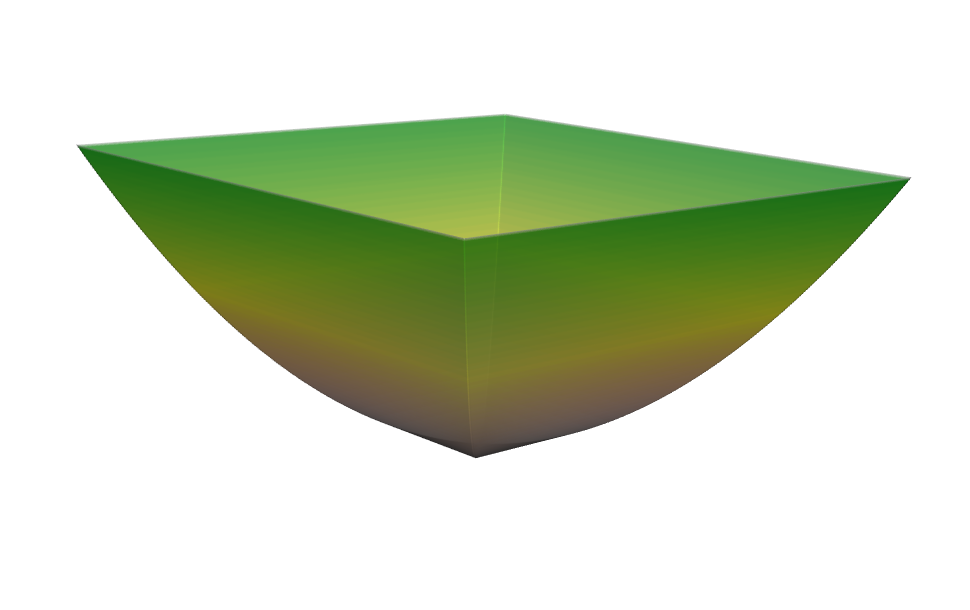} 
\end{tabular}
\end{center}
\caption{Left: The set $K$ from Example \ref{ex:bestsubset} is the union of the coordinate axes (red). The
  metric projection is the projection onto the closest coordinate
  axis. The exoskeleton of $K$ (blue) is the set of points $y = (y_1,
  y_2) \neq (0,0)$ with either $y_1 = y_2$ or $y_1 = -y_2$ for which
  the metric projection is not unique. The
  closure of the exoskeleton equals in this example the support of the
  singular measure $\nu$. Right: The convex function $\rho$ whose subgradient
  field contains the metric projection onto $K$.
  }
\end{figure}

To illustrate the general Theorem \ref{thm:main} we give a detailed
treatment of the case where $K$ is the union of two orthogonal
one-dimensional subspaces.  

\begin{ex} \label{ex:bestsubset} We consider the case $n = 2$, $\boldsymbol{\xi} = 0$, $\sigma^2 = 1$ 
and 
$$K = \{(y_1,y_2)\in\mathbb{R}^2\mid y_2 = 0\}
        \cup \{(y_1,y_2)\in \mathbb{R}^2\mid y_1 = 0\}$$
is the union of the two orthogonal subspaces formed by the first and
second coordinate axis. If we introduce the sets
$$I(z) = (-\infty, -|z|) \cup (|z|, \infty)$$ 
for $z \in \mathbb{R}$, we can for $y_1 \neq y_2$ write the metric projection as  
$$\mathrm{pr}(y_1, y_2) = (y_1 1_{I(y_2)}(y_1), y_2 1_{I(y_1)}(y_2)).$$
When $y_1 \neq y_2$ we find that 
$$\partial_1 \mathrm{pr}_1(\mathbf{y}) + \partial_2
\mathrm{pr}_2(\mathbf{y}) = 1_{I(y_2)}(y_1)  + 1_{I(y_1)}(y_2) = 1,$$
and $\mathrm{df}_S = 1$.  
To compute the singular measure $\nu$ we find, using Fubini's theorem and standard partial integration, that for $\phi \in C_c^1(\mathbb{R}^2)$,
\begin{eqnarray*}
\int_{\mathbb{R}^2}  \mathrm{pr}_1(\mathbf{y}) \partial_1 \phi(\mathbf{y}) \,
m_2(\mathrm{d}\mathbf{y}) & = & \int_{\mathbb{R}} \int_{I(y_2)} y_1 \partial_1 \phi(y_1,
y_2) \, \mathrm{d}y_1 \mathrm{d}y_2 \\
& = & - \int_{\mathbb{R}} |y_2|(\phi(-|y_2|, y_2) + \phi(|y_2|, y_2)) \mathrm{d}y_2 \\
&& \hspace{5mm} - \underbrace{\int_{\mathbb{R}} \int_{I(y_2)} \phi(y_1, y_2) \, \mathrm{d}y_1
  \mathrm{d}y_2}_{\int_{\mathbb{R}^2} \partial_1 \mathrm{pr}_1(y) \phi(y) \,
  \mathrm{d}m_2(y)}.
\end{eqnarray*}
This shows that the singular part of the distributional partial derivative of
$\mathrm{pr}_1(\mathbf{y})$ w.r.t. $y_1$ 
is the measure $\nu_1$ determined by  
$$\int_{\mathbb{R}^2} \phi(\mathbf{y}) \nu_1(\mathrm{d} \mathbf{y}) =  
 \int_{\mathbb{R}} |z| (\phi(|z|, z) + \phi(-|z|,z)) \mathrm{d} z.$$
The singular measure $\nu_2$ is determined likewise, and $\nu = \nu_1
+ \nu_2$ is given by  
$$\int_{\mathbb{R}^2} \phi(\mathbf{y}) \nu(\mathrm{d}\mathbf{y}) =  
 \int_{\mathbb{R}} |z| (\phi(|z|, z) + \phi(-|z|,z) + \phi(z,|z|) + \phi(z, -|z|))
 \mathrm{d} z.$$
By choosing positive functions $\phi_n \in C_c^1(\mathbb{R}^2)$ such
that $\phi_n(\mathbf{y}) \nearrow \psi(\mathbf{y}; \mathbf{0}, 1)$ for $n \to \infty$, it follows that 
$$\int_{\mathbb{R}^2} \psi(\mathbf{y}; \mathbf{0}, 1) \nu( \mathrm{d} \mathbf{y}) = \frac{2}{\pi} \int_{\mathbb{R}} |r| e^{-r^2}
\mathrm{d} r = \frac{2}{\pi} \int_0^{\infty} e^{-r} \mathrm{d} r =
\frac{2}{\pi}.$$
We find that the degrees of freedom for the selection among the two one-dimensional
orthogonal projections becomes
$$\mathrm{df} = 1 + \frac{2}{\pi} = 1.6366.$$
In this particular case it follows directly from the covariance
definition (\ref{eq:covdef}) that  
$$\mathrm{df} = E (\max\{X_1, X_2\})$$
where $X_1$ and $X_2$ are independent $\chi^2_1$-distributed random
variables. This concurs with findings in \cite{Ye:1998} on generalized
degrees of freedom. The numerical value could in this case also be computed by computing
the density of $\max\{X_1, X_2\}$, and use this to compute the
expectation $E (\max\{X_1, X_2\})$. 
\end{ex}

The example above corresponds to best subset selection in linear regression
with two orthogonal predictors. If we consider the general problem of 
best subset selection among subsets with $p$ linearly independent predictors we may note
that $\mathrm{df}_S = p$. Recently, \cite{Tibshirani:2014}
derived in the context of best subset selection an expression for 
$\mathrm{df} - p$ for orthogonal predictors, and developed some
generalizations of Stein's lemma as well. He coined
the term ``search degrees of freedom'' for the difference $\mathrm{df}
- p$, as this difference in the context of best subset selection
explicitly accounts for the contribution to the degrees of freedom
coming from the model search. A straightforward
consequence of our Theorem \ref{thm:main} is that the search degrees of
freedom is, in fact, always positive. A fact that is intuitively
reasonable -- and observable in applications and simulation studies
-- but it has to the best of our knowledge not been established
rigorously before. Though it may not be trivial, we expect that 
the measure $\nu$ can be computed for best subset selection in
general. This promises further insights into the costs that model
searching has on the degrees of freedom and ultimately the risk of the
estimator. 

As noted in the introduction, the risk estimate,
$\widehat{\mathrm{Risk}}$, given by (\ref{eq:riskid})
underestimates the true risk whenever $\mathrm{df} > \mathrm{df}_S$. An explicit representation of
the bias follows directly from Theorem \ref{thm:main}:
$$ E (\widehat{\mathrm{Risk}}) = \mathrm{Risk}
-  2  \sigma^2 \int_{\mathbb{R}^n} \psi(\mathbf{y}; \boldsymbol{\xi}, \sigma^2)
\nu(\mathrm{d} \mathbf{y}).$$
We observe that $\widehat{\mathrm{Risk}}$ is unbiased if and only if
the measure $\nu$ is the null measure. To control the size of the bias 
it may be useful to be able to bound the support of the singular
measure $\nu$. To this end we introduce the set of points with a non-unique metric
projection onto $K$. We call it the exoskeleton of $K$, following
the terminology in \cite{Hug:2004}, and we write 
$$\mathrm{exo}(K) = \left\{\mathbf{y} \in \mathbb{R}^n \, \middle| \, \argmin_{\mathbf{x} \in K} ||\mathbf{y} -
\mathbf{x}||_2^2 \textrm{ is not a singleton}\right\}.$$  
This set is also called the skeleton of the open set $K^{c}$ in
\cite{Fremlin:1997}.  Theorem \ref{thm:prdif} implies that 
$\mathrm{exo}(K)$ is a Lebesgue null set, but more is known. Theorem 1G
in \cite{Fremlin:1997} gives, for instance, that $\mathrm{exo}(K)$ has
Hausdorff dimension at most $n-1$. It should be noted that
there can be points in $K \backslash \mathrm{exo}(K)$ where
$\mathrm{pr}$ is not differentiable. We can then show the 
following proposition.

\begin{proposition} \label{prop:supp} If 
$$\mathrm{pr} : \mathbb{R}^n \backslash \overline{\mathrm{exo}(K)} \to
K$$ 
is locally Lipschitz, and in
  particular if it is $C^1$, then $\mathrm{supp}(\nu)
  \subseteq \overline{\mathrm{exo}(K)}$.
\end{proposition}

If $K$ is convex (in addition to being nonempty and closed) the
metric projection is uniquely defined everywhere and Lipschitz
continuous, see Lemma 1 in \cite{Tibshirani:2012}. Thus $\mathrm{exo}(K) =
\emptyset$ and by Proposition \ref{prop:supp} the measure $\nu$ is the
null measure. From this we get the unbiasedness of $\widehat{\mathrm{Risk}}$ for convex $K$.

\begin{corollary} \label{cor:Kconvex} The measure $\nu$ in Theorem
  \ref{thm:main} is the null measure if $K$ is convex, in which case 
the risk estimate $\widehat{\mathrm{Risk}}$ is unbiased.
\end{corollary} 

To illustrate the general results further we give two
additional examples. In Example \ref{ex:shrinkage} we consider the projection onto a convex $\ell_2$-ball, which
amounts to a form of $\ell_2$-shrinkage. In Example \ref{ex:3} we
consider the projection onto the $\ell_2$-sphere, which shows some
interesting phenomena in the non-convex case. Example \ref{ex:3} shows, in particular, that $K$
need not be convex for $\nu$ to be the null measure, and thus that the
support of $\nu$ can be a strict subset of
$\overline{\text{exo}(K)}$. 

\begin{ex} \label{ex:shrinkage} Let $K = B(\mathbf{0}, s)$ be the closed $\ell_2$-ball with center $\mathbf{0}$ and
  radius $s \geq 0$. Then 
$$\mathrm{pr}_i(\mathbf{y}) = \left\{\begin{array}{cc} 
\frac{s y_i}{||\mathbf{y}||_2} & \quad \mathrm{if } \ ||\mathbf{y}||_2 > s \\
y_i & \quad \mathrm{if } \ ||\mathbf{y}||_2 \leq s
\end{array}\right.
$$ 
and 
$$\partial_i \mathrm{pr}_i(\mathbf{y}) = \left\{\begin{array}{cc} 
\frac{s}{||\mathbf{y}||_2} - \frac{s y_i^2}{||\mathbf{y}||_2^3} & \quad \mathrm{if } \ ||\mathbf{y}||_2 > s \\
1 & \quad \mathrm{if } \ ||\mathbf{y}||_2 \leq s.
\end{array}\right.
$$ 
Since $K$ is convex 
$$\mathrm{df} = \mathrm{df}_S =  s(n-1) E (||\mathbf{Y}||_2^{-1}1(||\mathbf{Y}||_2 > s)) + n
P(||\mathbf{Y}||_2 \leq s).$$
If $\boldsymbol{\xi} = \mathbf{0}$ the expectation and probability can be expressed in terms of
incomplete $\Gamma$-integrals. The unbiased estimate of $\mathrm{df}$ is 
$$\nabla \cdot \mathrm{pr} (\mathbf{Y})  = \frac{s(n-1)}{||\mathbf{Y}||_2} 1(||\mathbf{Y}||_2 > s) + n 1(||\mathbf{Y}||_2
\leq s).$$
It is interesting to compare the constrained estimator, which for
fixed $s$ projects $\mathbf{Y}$ onto the ball of radius $s$, with the linear  
shrinkage estimator 
$$\frac{1}{1 + \lambda} \mathbf{Y}$$
for a fixed $\lambda \geq 0$. The linear shrinkage estimator coincides with the metric projection onto the
ball with radius 
\begin{equation} \label{eq:slamb}
s = ||\mathbf{Y}||_2/(1 + \lambda) \leq ||\mathbf{Y}||_2.
\end{equation}
It follows directly from (\ref{eq:covdef}) that the linear shrinkage estimator 
has degrees of freedom $n/(1+\lambda)$. 
For the metric projection onto a ball with radius $s$ given by
(\ref{eq:slamb}) the unbiased estimate of the degrees of freedom equals
$$\frac{s(n-1)}{||\mathbf{Y}||_2} = \frac{n-1}{1 + \lambda}.$$ 
This is an unbiased estimate of degrees of freedom for 
a ball with fixed radius $s \geq 0$. The degrees of freedom for the 
linear shrinkage estimator is for fixed $\lambda \geq 0$. The two estimates of degrees of
freedom differ because the relation $s(1+\lambda) = ||\mathbf{Y}||_2$ is
$\mathbf{Y}$-dependent.
\end{ex}

\begin{ex} \label{ex:3} In this example we take $K = S^{n-1}$ to be the $\ell_2$-sphere of radius
  $1$ in $\mathbb{R}^n$, and we take $\sigma^2 = 1$ and $\boldsymbol{\xi} = \mathbf{0}$. Then
  $\mathrm{pr}(\mathbf{y}) = \mathbf{y}/||\mathbf{y}||_2$ for $\mathbf{y} \neq \mathbf{0}$. The metric projection
  is not uniquely defined for $\mathbf{y} = \mathbf{0}$ and $\mathrm{exo}(S^{n-1}) = \{\mathbf{0}\}$. The computation of the
  divergence is as above with 
$$\nabla \cdot \mathrm{pr}(\mathbf{y}) = (n - 1) \frac{1}{||\mathbf{y}||_2}$$
for $\mathbf{y} \neq \mathbf{0}$. Since 
$$||\boldsymbol{\xi} - \mathrm{pr}(\mathbf{y})||_2^2 = \left|\left|
    \frac{\mathbf{y}}{||\mathbf{y}||_2}\right|\right|_2^2 = 1,$$
we find that $\mathrm{Risk} = 1$. Moreover, 
\begin{eqnarray*}
E||\mathbf{Y} - \mathrm{pr}(\mathbf{Y})||_2^2 & = & E \left(||\mathbf{Y}||_2^2 \left(1 -
  \frac{1}{||\mathbf{Y}||_2}\right)^2\right) \\
& = & E||\mathbf{Y}||_2^2 + 1 - 2 E ||\mathbf{Y}||_2 \\
& = & n + 1 -  2 E ||\mathbf{Y}||_2,
\end{eqnarray*}
and it follows that $\mathrm{df} = E ||\mathbf{Y}||_2$. Since $||\mathbf{Y}||_2^2 \sim
\chi^2_n$ straightforward computations give that 
$$E ||\mathbf{Y}||_2 = \frac{\sqrt{2} \Gamma\left(\frac{n +
      1}{2}\right)}{\Gamma\left(\frac{n}{2}\right)},$$
together with 
$$ E \left( \frac{1}{||\mathbf{Y}||_2}\right) = \frac{\Gamma\left(\frac{n -
      1}{2}\right)}{\sqrt{2} \Gamma\left(\frac{n}{2}\right)} = \frac{\sqrt{2} \Gamma\left(\frac{n +
      1}{2}\right)}{(n-1) \Gamma\left(\frac{n}{2}\right)}$$
for $n \geq 2$. This shows that 
$$\mathrm{df} = E ||\mathbf{Y}||_2 = (n-1) E \left( \frac{1}{||\mathbf{Y}||_2}\right) = E (\nabla
\cdot \mathrm{pr}(\mathbf{Y}))$$
for $n \geq 2$, and we conclude that $\nu$ is the null measure for $n \geq
2$.  This is an example where the measure $\nu$ can be
0 in cases where the exoskeleton is nonempty. 

For $n = 1$ we have $\mathrm{df} = E |Y| = \sqrt{\frac{2}{\pi}}$,
whereas $\mathrm{pr}(y) = \mathrm{sign}(y)$ has derivative $0$ for $y
\neq 0$, and thus $\mathrm{df}_S = 0$. It follows from Proposition
\ref{prop:supp} that $\nu = c \delta_0$ (with $\delta_0$ the Dirac
measure in 0) for $c \geq 0$. Since  
$$ \frac{c}{\sqrt{2 \pi}} = c  \psi(0; 0, 1)  = \int_{\mathbb{R}} \psi(y; 0, 1)
\nu(\mathrm{d} y) = 
 \sqrt{\frac{2}{\pi}}$$
we conclude that $\nu = 2 \delta_0$. 
Note that $\nu$ is the distributional derivative
of the sign function.
\end{ex}

\section{Divergence formulas for nonlinear least squares
  regression} \label{sec:nonlinear}

In this section our focus changes from the abstract results concerning
an arbitrary closed set $K$ in $\mathbb{R}^n$ to sets that are given
in terms of a $p$-dimensional parametrization. The main purpose is to provide explicit
formulas for the computation of the divergence  $\nabla \cdot
\mathrm{pr}(\mathbf{y})$ for a given $\mathbf{y} \in \mathbb{R}^n$ in terms of the 
parametrization in two different situations of practical interest. 
Both results follow by implicit differentiation. The complete proofs
are given in Section 2 in the supplementary material. 
 
We assume in this section that  $\zeta : \mathbb{R}^p \to
\mathbb{R}^n$, that $\Theta \subseteq
\mathbb{R}^p$ is a closed set, and that the image $K = \zeta(\Theta)$
is closed. The observation $\mathbf{y} \in \mathbb{R}^n$ is fixed, and we
make the following local regularity assumptions about the
parametrization $\zeta$.
\begin{itemize}
\item The metric projection of $\mathbf{y}$ onto $K$ is unique with 
$\mathrm{pr}(\mathbf{y}) = \zeta(\hat{\beta})$ for $\hat{\beta} \in \Theta$. 
\item The map $\zeta : \mathbb{R}^p \to \mathbb{R}^n$ is $C^2$ in a
  neighborhood of $\hat{\beta}$.
\item The map $\zeta : \Theta \to K$ is open in $\hat{\beta}$, that is, if $V$ is a neighborhood of $\hat{\beta}$ in
  $\mathbb{R}^p$, there is a neighborhood $U$ of $\mathrm{pr}(\mathbf{y})$ in
  $\mathbb{R}^n$ such
  that $$U \cap K \subseteq \zeta(V \cap \Theta).$$
\end{itemize}
The inverse function
theorem implies the last assumption if the derivative of $\zeta$ has
rank $p$ (forcing $p \leq n$) in $\hat{\beta}$. 

We introduce the two $p \times p$ matrices
$G$ and $J$ by
\begin{equation} \label{eq:Gdfn}
G_{kl} = \sum_{i=1}^n \partial_k \zeta_i(\hat{\beta}) \partial_l
\zeta_i(\hat{\beta})
\end{equation}
and 
\begin{equation} \label{eq:Jdfn}
J_{kl} = G_{kl} - \sum_{i=1}^n (y_i -
\zeta_i(\hat{\beta})) \partial_k \partial_l \zeta_i(\hat{\beta}).
\end{equation}
Note that for a linear model where $\zeta(\beta) = \mathbf{X}
\beta$ for an $n \times p$ matrix $\mathbf{X}$, $J = G = \mathbf{X}^T\mathbf{X}$.

\begin{theorem} \label{thm:divsmooth} If $\hat{\beta} \in
  \Theta^{\circ}$ and $J$ has full rank $p$, then 
$$\nabla \cdot \mathrm{pr} (\mathbf{y}) = \mathrm{tr} \left(J^{-1} G \right).$$
\end{theorem}

Note that under sufficient regularity assumptions, standard asymptotic
arguments, see Sections 2.3 and 2.5 in \cite{Claeskens:2008}, give for
$p$ fixed the expansion
$$||\mathbf{Y} - \mathrm{pr}(\boldsymbol{\xi})||^2_2 = ||\mathbf{Y} - \mathrm{pr}(\mathbf{Y})||^2_2 + Z + 2 \sigma^2 U^T \mathbb{J}^{-1} U + o_P(1)$$
for $n \to \infty$, with $EZ = 0$, $EU = 0$, $VU = \mathbb{G}$,
$$\mathbb{G}_{kl} = \sum_{i=1}^n \partial_{k} \zeta_i(\beta_0) \partial_l
\zeta_i(\beta_0)  \quad \mathrm{and} \quad \mathbb{J}_{kl} =
\mathbb{G}_{kl} - \sum_{i=1}^n (\boldsymbol{\xi}_i - \mathrm{pr}_i(\boldsymbol{\xi})) \partial_k
\partial_l \zeta_i(\beta_0).$$
The parameter $\beta_0$ is defined by $\zeta(\beta_0) =
\mathrm{pr}(\boldsymbol{\xi})$, that is, $\zeta(\beta_0)$ is the point in the model $K
= \zeta(\Theta)$ closest to $\boldsymbol{\xi}$. Defining $p^* = E(U^T \mathbb{J}^{-1}
U) = \mathrm{tr} (\mathbb{J}^{-1} \mathbb{G})$ as the effective number
of parameters, the generalization of AIC to misspecified models, known as Takeuchi's information criterion, becomes
$$\mathrm{TIC} =  ||\mathbf{y} - \mathrm{pr}(\mathbf{y})||^2_2 + 2\sigma^2 p^*.$$
We recognize $J$ and $G$ as plug-in estimates of $\mathbb{J}$ and
$\mathbb{G}$, and thus $\mathrm{tr} \left(J^{-1} G \right)$ as an
estimate of $p^*$. Theorem \ref{thm:divsmooth} identifies this
estimate as the unbiased estimate of the Stein degrees of
freedom. From the asymptotic arguments it does not follow that
$\mathrm{TIC}$ is negatively biased for finite sample sizes, 
but our Theorem \ref{thm:main} reveals that $p^*$ generally needs a
finite sample correction.

We then turn our attention to the case where the parameter set is an
$\ell_1$-constrained subset of $\mathbb{R}^p$. That is, we 
consider parameter sets of the form
$$\Theta_s = \left\{ \beta \in \mathbb{R}^p \;\middle|\; \sum_{k=1}^p
\omega_k |\beta_k|  \leq s\right\}$$
for $s \geq 0$ and $\omega \in \mathbb{R}^p$ a fixed vector of
nonnegative weights. With $\mathrm{pr}(\mathbf{y}) = \zeta(\hat{\beta})$ for
$\hat{\beta} \in \Theta_s$, then $\hat{\beta}$ is typically on the
boundary of $\Theta_s$, and the formula in Theorem \ref{thm:divsmooth}
for the divergence does not apply. Instead we note that $\hat{\beta}$
fulfills the Karush-Kuhn-Tucker conditions 
$$D\zeta(\hat{\beta})^T (\mathbf{y} - \zeta(\hat{\beta})) = \hat{\lambda} \gamma$$
for $\gamma \in \mathbb{R}^p$ with 
$$\begin{array}{ll} \gamma_k = \omega_k \mathrm{sign}(\hat{\beta}_k)
    \quad & \text{if } \ \hat{\beta}_k \neq 0 \\
\gamma_k \in [-\omega_k,\omega_k] \quad & \text{if } \ \hat{\beta}_k = 0 
\end{array} $$
and $\hat{\lambda} \geq 0$ the Lagrange multiplier. We introduce the active set of parameters as 
$$\mathcal{A} = \{ i \mid \hat{\beta}_i \neq 0 \},$$
and let $J_{\mathcal{A}, \mathcal{A}}$ and $G_{\mathcal{A},
  \mathcal{A}}$
denote the submatrices of $J$ and $G$, respectively, with indices
in $\mathcal{A}$. 

\begin{dfn} A solution to the Karush-Kuhn-Tucker conditions is said to
  fulfill the sufficient second order conditions if $\hat{\lambda} > 0$, 
$\gamma_k \in (-\omega_k, \omega_k)$ for $k \not \in \mathcal{A}$ and 
$\delta^T J_{\mathcal{A}, \mathcal{A}} \delta > 0$ for all nonzero $\delta \in
\mathbb{R}^{\mathcal{A}}$ satisfying $\delta^T \gamma_{\mathcal{A}} = 0$.  
\end{dfn}

Note that the sufficient second order conditions imply that a solution to the Karush-Kuhn-Tucker conditions 
is a local minimizer of $||\mathbf{y} - \zeta(\beta)||_2^2$ in
$\Theta_s$. 

\begin{theorem} \label{thm:divlasso} If $J_{\mathcal{A}, \mathcal{A}}$ has full rank
  $|\mathcal{A}|$, if $\gamma_{\mathcal{A}}^T (J_{\mathcal{A}, \mathcal{A}})^{-1}
  \gamma_{\mathcal{A}} \neq 0$ and if $\hat{\beta}$ fulfills the sufficient second order 
  conditions, then
$$\nabla \cdot \mathrm{pr} (\mathbf{y}) = \mathrm{tr} \left((J_{\mathcal{A},
    \mathcal{A}})^{-1}G_{\mathcal{A}, \mathcal{A}}\right) - \frac{ \gamma_{\mathcal{A}}^T (J_{\mathcal{A},
    \mathcal{A}})^{-1}  G_{\mathcal{A}, \mathcal{A}} (J_{\mathcal{A},
    \mathcal{A}})^{-1}  \gamma_{\mathcal{A}}}{\gamma_{\mathcal{A}}^T (J_{\mathcal{A},
    \mathcal{A}})^{-1} \gamma_{\mathcal{A}}}.$$
\end{theorem}

First note that $J_{\mathcal{A}, \mathcal{A}}$ has full rank
  $|\mathcal{A}|$ and $\gamma_{\mathcal{A}}^T (J_{\mathcal{A}, \mathcal{A}})^{-1}
  \gamma_{\mathcal{A}} \neq 0$ if $J_{\mathcal{A}, \mathcal{A}}$ is positive
  definite. For the linear model, this is the case when
  $\mathbf{X}_{\cdot, \mathcal{A}}$ has rank $|\mathcal{A}|$. 
Then observe that in the case where $\zeta$ is locally linear around
$\hat{\beta}$ to second order, that is, $\partial_k\partial_l \zeta(\hat{\beta}) = 0$,
we get that $\nabla \cdot \mathrm{pr} (\mathbf{y}) = |\mathcal{A}| -
1$. Previous results in \cite{Zou:2007} and
\cite{Tibshirani:2012} for $\ell_1$-penalized linear regression give
that the unbiased estimate of degrees of freedom is
$|\mathcal{A}|$. The difference arises because we consider the
constrained estimator, and this phenomenon was first observed in
\cite{Kato:2009}. See also Example \ref{ex:shrinkage} for a similar
difference for $\ell_2$-regularization. It is
possible to compute the divergence of the penalized estimator under
conditions similar to those above. The result is
$\mathrm{tr} \left((J_{\mathcal{A}, \mathcal{A}})^{-1}G_{\mathcal{A},
    \mathcal{A}}\right)$ as expected. However, we cannot in an obvious
way relate this quantity to the degrees of freedom of the penalized
nonlinear least squares estimator. Our results hinge crucially on the
fact that the estimator can be expressed in terms of a metric
projection onto a closed set. If the penalized estimator can be given
such a representation, e.g. via dualization as outlined in
\cite{Tibshirani:2012} in the linear case, we might be able to
transfer the results to the penalized estimator, but we expect this to
be difficult without convexity.

\section{Model selection for a $d$-dimensional linear ODE} \label{sec:ode}

In this section we present simulation results on the use of the risk estimate
$\widehat{\mathrm{Risk}}$ based on the divergence in a
nontrivial example of nonlinear regression. The
example considered is estimation of the parameters in a system of
linear ordinary differential equations using an $\ell_1$-constrained 
estimator as well as a model search approach. The main conclusion is 
that the bias of $\widehat{\mathrm{Risk}}$ was
considerable for the model search, while it was negligible for the 
$\ell_1$-constrained estimator. All computations were carried
out using the R package \texttt{smde}, see
\url{http://www.math.ku.dk/~richard/smde/}. 

The observations are $\mathbf{Y}_1, \ldots, \mathbf{Y}_m \in \mathbb{R}^d$ with $\mathbf{Y}_i \sim
\mathcal{N}(\boldsymbol{\xi}_i, \sigma^2 I_d)$ and $\boldsymbol{\xi}_i = e^{t_i B} \mathbf{x}_i$ for
$t_i > 0$, $\mathbf{x}_i \in \mathbb{R}^d$ and $e^{tB}$ denoting the matrix
exponential. It is well known that $t \mapsto e^{tB}\mathbf{x}$ is the solution
of the linear $d$-dimensional ODE
$$\frac{\mathrm{d}}{\mathrm{d} t} f(t) = B f(t)$$
for $t > 0$ with initial condition $f(0) =  \mathbf{x} \in \mathbb{R}^d$.
The unknown parameter is $B \in \mathbb{M}(d, d)$.
We collect the observations into 
$\mathbf{Y} = (\mathbf{Y}_1, \ldots, \mathbf{Y}_m) \in \mathbb{M}(d, m)$,
and we let likewise $\boldsymbol{\xi} = (\boldsymbol{\xi}_1, \ldots, \boldsymbol{\xi}_m)$ denote the collection of
expectations. We will identify the matrices $\mathbf{Y}$ and $\boldsymbol{\xi}$ with
vectors in $\mathbb{R}^n$ for $n = md$, which we denote by $\mathbf{Y}$ and $\boldsymbol{\xi}$
as well (formally, the identification is made by stacking the columns). Thus $\mathbf{Y} \sim
\mathcal{N}(\boldsymbol{\xi}, \sigma^2 I_n)$. We also identify $B$ with a vector in
$\mathbb{R}^p$ where $p = d^2$, and the parametrization $\zeta : \mathbb{R}^p \to
\mathbb{R}^n$ is given as 
\begin{equation} \label{eq:expmean}
\zeta(B) = (e^{t_1B}\mathbf{x}_1, \ldots, e^{t_m B} \mathbf{x}_m).
\end{equation}
We note that the number of observations $n = md$ as well as
the number of parameters $p = d^2$ scale with $d$. For
many applications it may be realistic to achieve a good
model for a sparse $B$. Sparse
estimation of $B$ is generally useful for computational and statistical reasons, and 
it may also be useful for network inference and interpretations.

We will in this paper focus on the special case $t_1 = \ldots = t_m
= t$, which we will refer to as the isochronal model. For the
isochronal model $\zeta(B) = e^{tB}\mathbf{x}$ with $\mathbf{x} = (\mathbf{x}_1, \ldots,
\mathbf{x}_m)$, in which case it is natural to parametrize the model in terms of $A =
e^{tB}$. With $\hat{A}$ an estimator of $A$ we can estimate $B$
as $\hat{B} = \log(\hat{A})/t$ where $\log$ denotes the
principal matrix logarithm. The least squares estimator of $A$ amounts
to ordinary linear least squares regression. We are, however,
interested in obtaining sparse estimates of
$B$. Since the principal matrix logarithm does not preserve sparseness
in general, we will maintain the parametrization in terms of $B$ and
consider the family of $\ell_1$-constrained nonlinear least squares estimators 
$$ \hat{B}_s = \argmin_{B \in \Theta_s} ||\mathbf{Y} - e^{tB} \mathbf{x}||_2^2$$ 
where $\Theta_s = \{B \mid \sum_{kl} \omega_{kl} |B_{kl}| \leq s\}$
for $s \geq 0$ and $\omega \in \mathbb{M}(d,d)$ is a given weight matrix (with
$\omega_{kl} \geq 0$). For technical details on the computation of derivatives
and the implementation of the optimization algorithm see the supplementary material.

\begin{figure}[t]
\begin{center}
\includegraphics[width=0.9\textwidth]{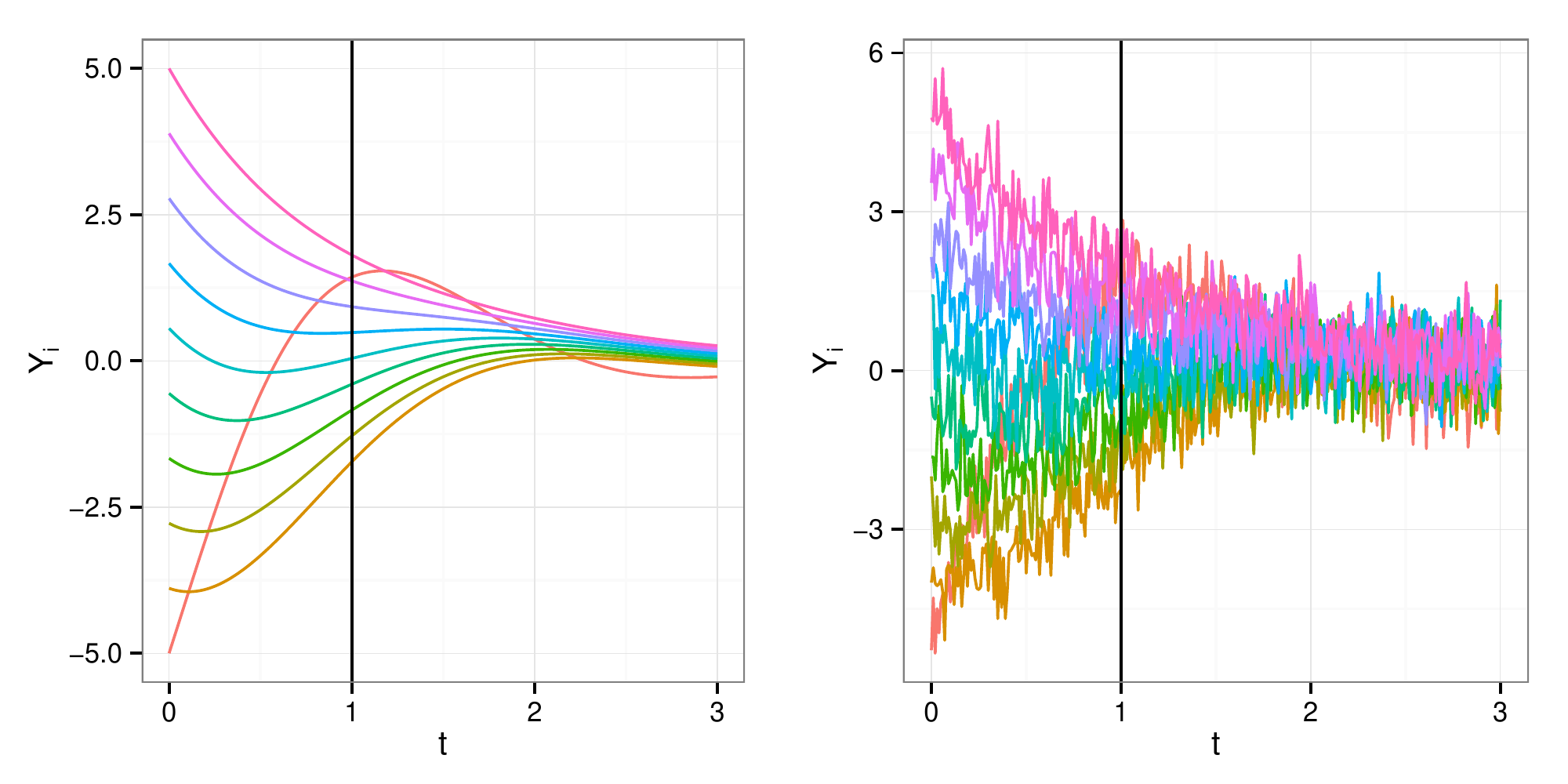}
\end{center}
\caption{Example of the solution of the ODE and a noisy sample path. \label{fig:sim}}
\end{figure}

We did a simulation study with
$t = 1$, $d = 10$, $m = 15$ and $\sigma^2 = 0.25$. The $B$ matrix
is given in the supplementary material, and contains 28 nonzero parameters out of
100. The matrix $B$ was chosen so that $e^B$ is dense and not well approximated
by a sparse matrix. The matrix exponential $e^B$ is, in particular,
not well approximated by the first order Taylor approximation $I_{10} +
B$. A single simulation of the sample paths is shown in Figure \ref{fig:sim}.

The initial conditions were sampled from the
$10$-dimensional normal distribution $\mathcal{N}(0, 16 I_{10})$, and we
used a total of $1000$ replications. For the choice of
weights (the $\omega_{kl}$'s) we considered two situations; either
$\omega_{kl} = 1$, or adaptive weights, as introduced in \cite{Zou:2006b}, based on the MLE, 
$$\omega_{kl} = \frac{1}{|\hat{B}_{kl}|}.$$ 
In this section we only report the results for the unit weights. See
the supplementary material for the results using adaptive weights. 

In the simulation study we computed the $\ell_1$-constrained estimators
$\hat{B}_s$ for a range of values of $s$ and the corresponding estimates
$\widehat{\mathrm{Risk}}(s)$ of the risk based on
(\ref{eq:riskest}). The divergence was computed using 
Theorem \ref{thm:divlasso} based on the formulas in Section 3 in
the supplementary material. 
\begin{figure}[t]
\begin{center}
\includegraphics[width=0.6\textwidth]{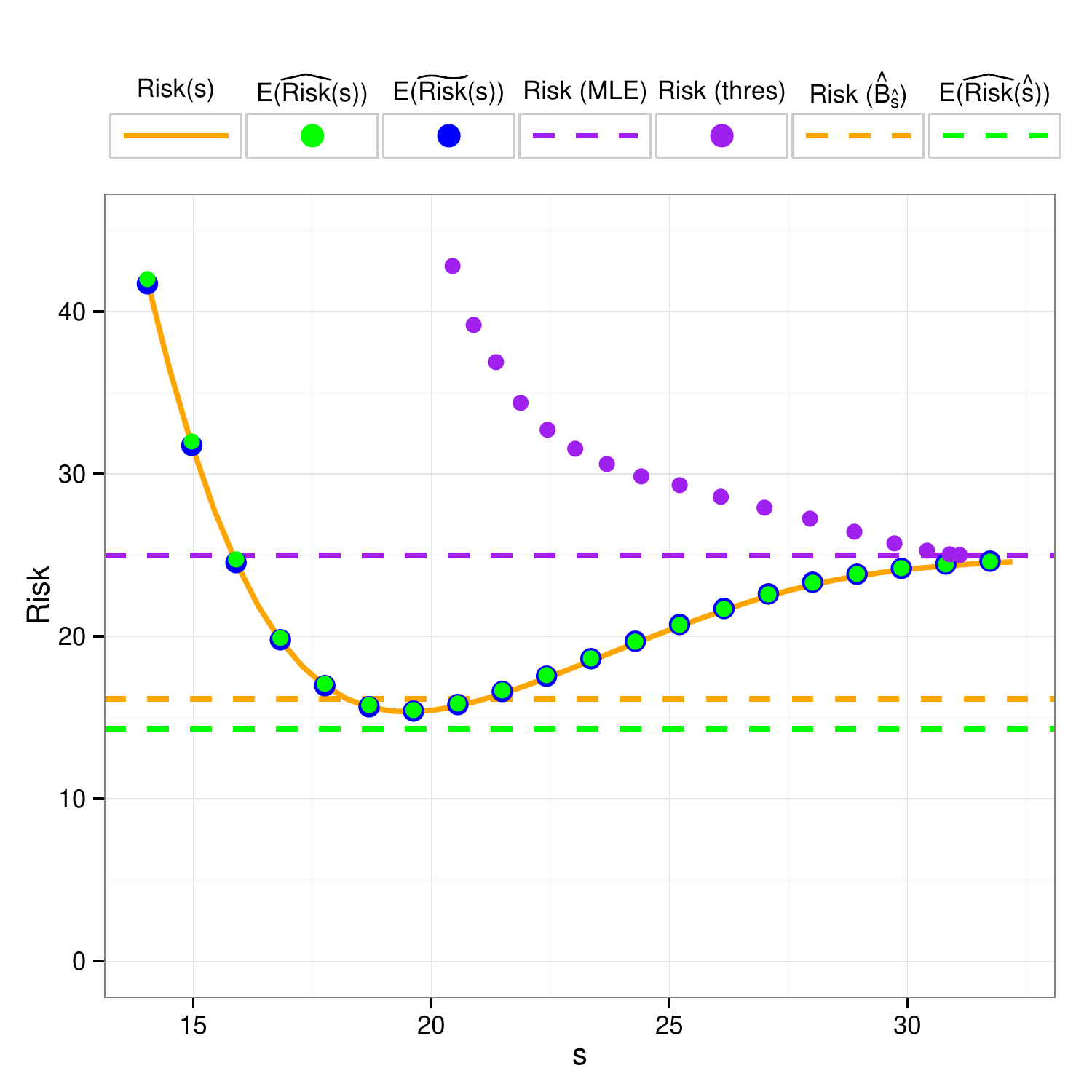} 
\end{center}
\caption{Risks for the $\ell_1$-constrained estimator with unit weights
  as a function of the constraint $s$ compared to the risk of
  the MLE and hard thresholding of the MLE. In addition, expected
  values of risk estimates. The 
risk estimates were practically unbiased for the $\ell_1$-constrained
estimator. \label{fig:risks1}}
\end{figure}
With 
$$\hat{s} = \argmin_s \widehat{\mathrm{Risk}}(s)$$
denoting the data driven optimal estimate of $s$, the resulting estimator of $B$
is $\hat{B}_{\hat{s}}$. In addition, we computed the risk estimate
$$\widetilde{\mathrm{Risk}}(s) = ||\mathbf{Y} - \mathrm{pr} (\mathbf{Y})||^2_2 - n \sigma^2
+ 2 \sigma^2 (|\mathcal{A}| - 1)$$
based on the approximation $\nabla \cdot e^{\hat{B}_s} \mathbf{x} \simeq
|\mathcal{A}| - 1$, see the discussion after Theorem
\ref{thm:divlasso}. We also computed the
MLE as well as a sequence of sparse(r) solutions obtained by hard
thresholding the MLE. The results of the simulation study are
summarized in Figure \ref{fig:risks1}. The risk of the constrained estimator was minimal
around $s = 19.5$. Both risk estimates,
$\widehat{\mathrm{Risk}}(s)$ and $\widetilde{\mathrm{Risk}}(s)$, were, in this case,
very close to being unbiased, and the estimated optimal constrained
$\hat{s}$ gave an estimator $\hat{B}_{\hat{s}}$ with close to minimal
risk. The MLE and the sequence of thresholded MLEs all have larger risks 
than the constrained estimators for a substantial range of $s$, and, 
more importantly, than the risk of $\hat{B}_{\hat{s}}$. We
should note, however, that $\widehat{\mathrm{Risk}}(\hat{s})$ did
on average underestimate the actual risk of $\hat{B}_{\hat{s}}$ a little. 

\begin{figure}[t]
\begin{center}
\includegraphics[width=0.6\textwidth]{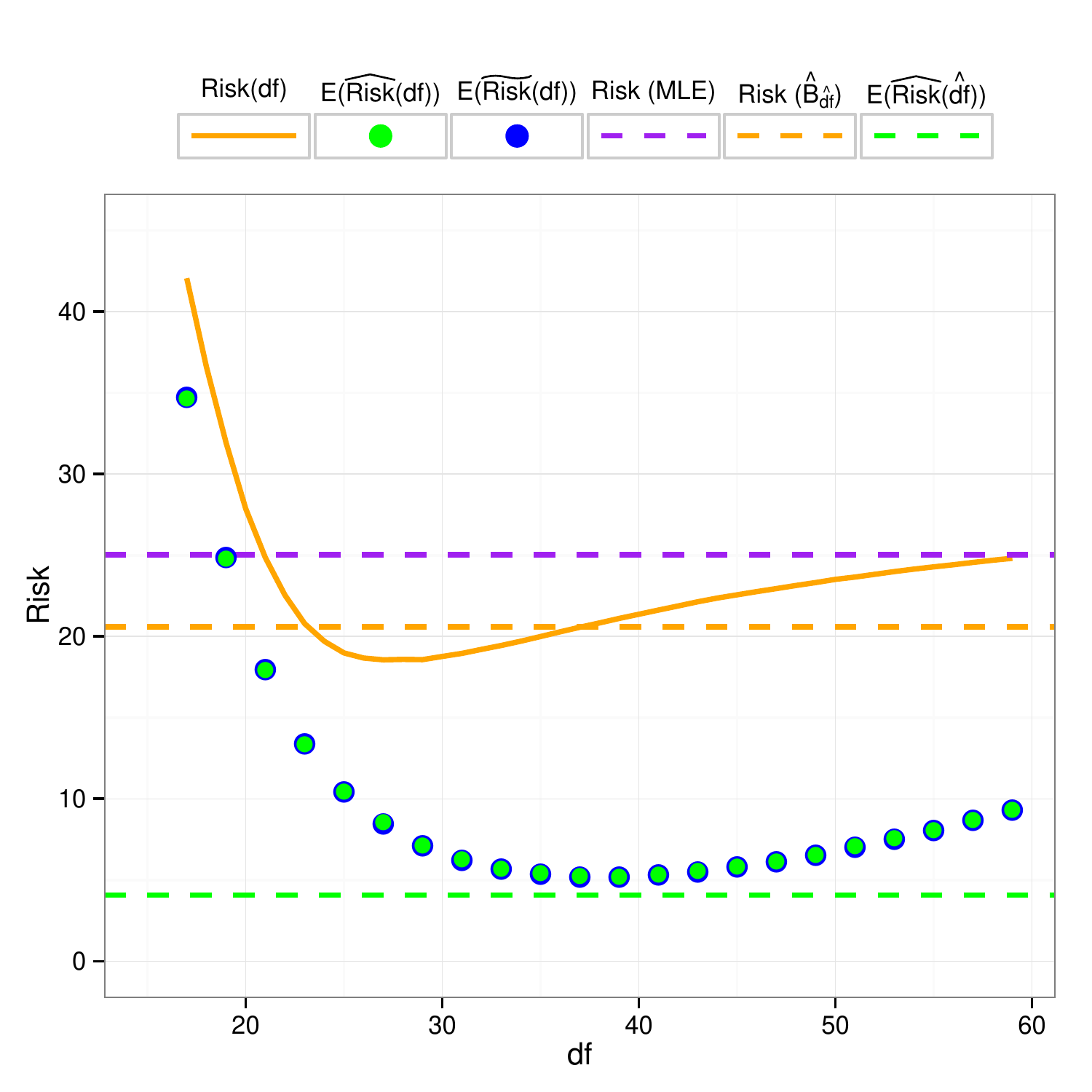} 
\end{center}
\caption{Risks for the forward stepwise model search as a function
  of the number of nonzero parameters compared to the risk of
  the MLE. In addition, expected
  values of risk estimates. The 
risk estimates grossly underestimated the true risk and overestimated
the optimal number of nonzero parameters. \label{fig:risks2}}
\end{figure}

In addition to the $\ell_1$-constrained estimator, we considered
classical model searching. That is, we sought the best fitting model
among all models with a given number of nonzero parameters. A complete
search is computationally prohibitive, so we carried out a forward
stepwise model search. The model search was initiated by a
diagonal matrix, and in each step we added the parameter that
decreased the squared error loss the most. The divergences
 were computed using either Theorem \ref{thm:divsmooth} or
approximated by the number of nonzero parameters. The results are
summarized in Figure \ref{fig:risks2}. We found that the model with
minimal risk had around 28 nonzero parameters. In this case, the risk
estimates underestimated the true risk considerably. Moreover, they 
suggested that models with around 37 nonzero parameters had minimal
risk. Consequently, the data driven choice of the number of nonzero
parameters resulted in too large models with a correspondingly larger
risk. In contrast to the $\ell_1$-constrained case, the integral w.r.t. the
singular measure in Theorem \ref{thm:main} can be detected as a bias for
model searching. This bias cannot be ignored if we want to estimate the risk
satisfactorily.

To understand better the results of the simulation study -- and the
nature of the nonlinear least squares problem -- it would be desirable to
be able visualize the image sets $\exp(\Theta_s)$, or, in particular, the
images $\exp(\partial\Theta_s)$ of the boundaries of $\Theta_s$, for
different choices of $s$. These are the images under the matrix
exponential of the boundaries of $\ell_1$-balls. As these sets
are subsets of $\mathbb{R}^{100}$ a visualization is challenging. Figure
\ref{fig:slices} shows two selected slices of the sets by
affine subspaces. The slices were constructed as follows. With 
 
$$
e^{B(a,b,c,d)} = \left( \begin{array}{ccc}
 a & c & * \\ 
  b & d & * \\ 
 * & * & *
\end{array}\right)
$$ 
it holds that $B(-0.11, 0.19, -0.19, 0.23) = B$  -- the matrix that we used in the
simulation. Fixing either $(b, c) = (0.19, -0.19)$ or $(a,d) = (-0.11, 0.23)$
we get the two affine subspaces considered, which both include
$B$. The slices were computed as
contour curves for \mbox{$(a,d) \mapsto ||B(a, 0.19, -0.19,
d)||_1$} and $(c,d) \mapsto ||B(-0.11, c, d, 0.23)||_1$.

\begin{figure}
\begin{center}
\includegraphics[width=0.9\textwidth]{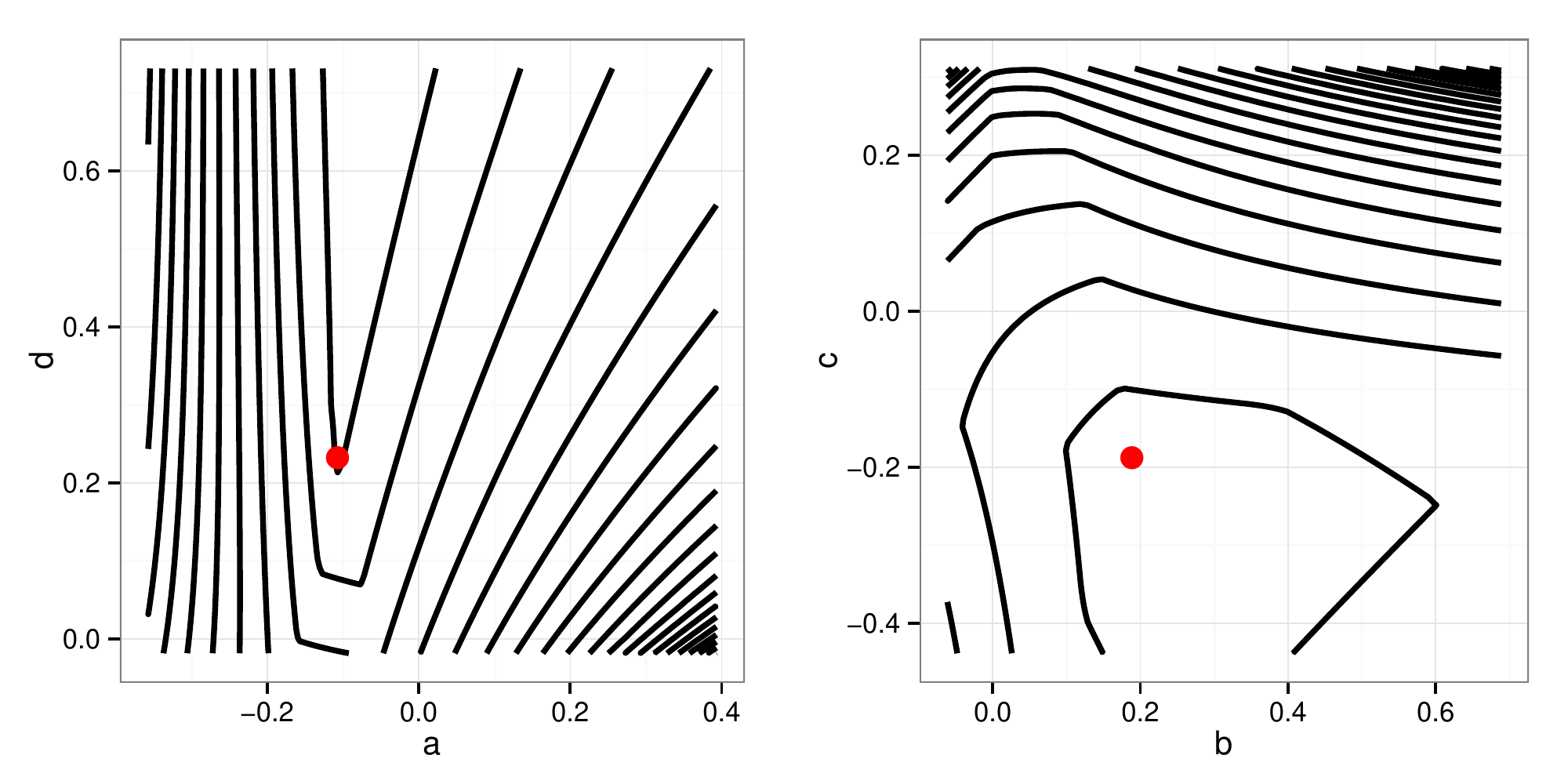}
\end{center}
\caption{Two selected 2-dimensional slices of $\exp(\partial \Theta_s)$
  for different choices of $s$, that is, the intersections of $\exp(\partial \Theta_s)$ in
  $\mathbb{R}^{100}$ with 2-dimensional affine subspaces. The red points
mark the values used in the simulation study. \label{fig:slices}}
\end{figure}

\section{Proofs} \label{sec:proofs} In this section we give the proofs of the results
stated in Sections \ref{sec:df}. Doing so we
will provide a brief account on the ideas and strategies used with
some appropriate references to the literature. A further discussion of how our
results and proofs are related to the literature is given in Section
\ref{sec:dis}. 

The proofs are based on the facts that the function $\rho$ defined by
(\ref{eq:rhodf}) is convex, that its
subgradient in $\mathbf{y}$ contains the points in $K$ closest to
$\mathbf{y}$, and that if $\rho$ is differentiable in $\mathbf{y}$, its
gradient equals the necessarily unique metric projection. That is, 
$\mathrm{pr}(\mathbf{y}) = \nabla \rho (\mathbf{y})$. 
This is all well known, see e.g. Theorem 3 in
\cite{Asplund:1968} for a similar but abstract formulation, or Theorem 3.3 in
\cite{Evans:1987} for an alternative formulation in
$\mathbb{R}^n$. For completeness, Lemma 1 and its proof in the supplementary material
give the details. Central to the proofs of Theorem \ref{thm:prdif} 
and Theorem \ref{thm:main} in Section \ref{sec:df} is a famous theorem of Alexandrov given
first in \cite{Alexandrov:1939}. It loosely states that a
convex function is twice differentiable except perhaps on a
Lebesgue null set. We state a version of Alexandrov's theorem particularly useful
for our purposes, which we will apply to the convex function
$\rho$.  

\begin{theorem} \label{thm:alexandrov}
  Let $g : \mathbb{R}^n \to \mathbb{R}$ be a convex function, and let
  $D \subseteq \mathbb{R}^n$ denote the subset on which $g$ is
  differentiable. For Lebesgue almost all $y$ it holds that $y \in D$
  and there exists a matrix $A$ such that 
\begin{equation} \label{eq:firstorder}
\nabla g (\mathbf{x}) = \nabla g(\mathbf{y}) + A(\mathbf{x} -
\mathbf{y}) + o(||\mathbf{x} - \mathbf{y}||_2)
\end{equation}
for $\mathbf{x} \in D$. The matrix $A$ is symmetric and positive semidefinite and as such
uniquely determined by (\ref{eq:firstorder}). 
\end{theorem}

The theorem is a direct consequence of Theorem 2.3 and Theorem 2.8 in 
\cite{Rockafellar:2000}. See, in addition, 
Chapter 13 -- and Theorem 13.51 in particular -- in
\cite{Rockafellar:1998} for similar results. Theorem \ref{thm:alexandrov} also follows from
Theorem 6.1 and Theorem 7.1 in \cite{Howard:1998}, which is a nice
self contained exposition of Rademacher's and Alexandrov's
theorems. 

In the light of
Definition \ref{dfn:dif}, Theorem
\ref{thm:alexandrov} says that for a convex function $g$, $\nabla g$ is defined
Lebesgue almost everywhere, and $\nabla g$ is differentiable in the
extended sense Lebesgue almost everywhere. Note, however, that the
differentiability points of $\nabla g$ can be a strict subset of 
its maximal domain of definition. 

\begin{proof}[Proof of Theorem \ref{thm:prdif}] The existence of  
a Borel measurable selection follows from general results in
\cite{Rockafellar:1998}. 
The set valued metric projection 
$\mathrm{Pr}$ is defined as 
$$\mathrm{Pr}(\mathbf{y}) = \argmin_{\mathbf{x} \in K} ||\mathbf{y} - \mathbf{x}||^2_2.$$ 
 As a set valued map,
$\mathrm{Pr}$ is outer semicontinuous by Example 5.23 in 
  \cite{Rockafellar:1998}, and combining Theorem 5.7 and Exercise 14.9
  in \cite{Rockafellar:1998} it is, still as a set valued map,
  closed-valued and Borel measurable. Corollary 14.6 in
  \cite{Rockafellar:1998} implies that $\mathrm{Pr}$ admits a
  Borel measurable selection, that is, there is a Borel measurable map 
$\mathrm{pr} : \mathbb{R}^n \to \mathbb{R}^n$ with
$$\mathrm{pr}(\mathbf{y}) \in \mathrm{Pr}(\mathbf{y})$$
for all $\mathbf{y} \in \mathbb{R}^n$. 

Alexandrov's Theorem can then be used to show that the selection of
$\mathrm{pr}(\mathbf{y})$ is unique and differentiable in the extended sense
for Lebesgue almost all $\mathbf{y}$. Theorem
  \ref{thm:alexandrov} holds for the convex function $\rho$. 
  For those $\mathbf{y}$ where (\ref{eq:firstorder}) holds, the
  differentiability of $\rho$ in $\mathbf{y}$ assures that $\mathrm{pr}(\mathbf{y}) =
  \nabla \rho(\mathbf{y})$ is   uniquely defined in $\mathbf{y}$ as well as differentiable in $\mathbf{y}$ in the
  sense of (\ref{eq:firstorder}). The domain $D$ on which $\mathrm{pr}$
  is uniquely defined thus satisfies that $D^c$ is a Lebesgue null
  set, and $\mathrm{pr} : D \mapsto \mathbb{R}^n$ satisfies
  (\ref{eq:firstorder}) for Lebesgue almost all $\mathbf{y}$. That is,
$$\mathrm{pr} (\mathbf{x}) = \mathrm{pr}(\mathbf{y}) + A(\mathbf{x}- \mathbf{y}) + o(||\mathbf{x}-\mathbf{y}||_2)$$
for $\mathbf{x} \in D$, and $\mathrm{pr}$ is differentiable in the
extended sense for Lebesgue almost all $\mathbf{y}$. By definition,
$$\partial_j \mathrm{pr}_i(\mathbf{y}) = A_{ij}$$
for those $\mathbf{y}$ where $\mathrm{pr}$ is differentiable in the extended sense, and since $A$
is positive semidefinite, $\partial_i \mathrm{pr}_i(\mathbf{y}) \geq 0$ for $i
= 1, \ldots, n$. 
\end{proof}

From hereon we assume, in accordance with Theorem \ref{thm:prdif}, that 
a choice of $\mathrm{pr}$ has been made on the set where
$\mathrm{pr}$ is not unique, such that $\mathrm{pr} : \mathbb{R}^n \to
\mathbb{R}^n$ is Borel measurable. 

We turn to the proof of Theorem \ref{thm:main}. The relation in
Theorem \ref{thm:main} between the degrees of freedom, $\mathrm{df}$, and the
Stein degrees of freedom, $\mathrm{df}_S$, will be established by partial
integration. However, to handle metric projections in full generality 
we have to turn to distributional formulations of differentiation. Partial
integration holds by definition for distributional differentiation. What we need
is to identify the distributional partial derivatives of the coordinates of the metric
projection. For this purpose, we define a signed Radon measure to be the difference
of two (positive) Radon measures. In this sense a signed Radon
measure need not have bounded total variation. Though we have to be
careful with such a definition to avoid the undefined ``$\infty -
\infty$", the difference of two Radon measures does give a well
defined linear functional on $C_c(\mathbb{R}^n)$.   

\begin{dfn} \label{dfn:lbv} A function $g \in L^1_{\text{loc}}(\mathbb{R}^n)$ is of locally bounded
  variation if there exist signed Radon measures $\mu_j$ for $j = 1,
  \ldots, n$ on $\mathbb{R}^n$ such that 
$$\int_{\mathbb{R}^n} g(\mathbf{y}) \partial_j \phi(\mathbf{y}) \, \mathrm{d}\mathbf{y} = - \int_{\mathbb{R}^n} \phi(\mathbf{y}) \mu_j(\mathrm{d}\mathbf{y})$$
for all $\phi \in C_c^{\infty}(\mathbb{R}^n)$.
\end{dfn}

Thus the functions of locally bounded variation are those
$L^1_{\text{loc}}$-functions whose distributional partial derivatives are
signed Radon measures. It is easily verified that Definition
\ref{dfn:lbv} is equivalent to other definitions in the literature,
e.g. the definition in Chapter 5 in \cite{Evans:1992}.

\begin{lemma} \label{lem:bv}
The functions $\mathrm{pr}_i$ for $i = 1, \ldots, n$ are of locally bounded
  variation. With $\mu_{ij}$ denoting the $j$'th distributional partial derivative of
  $\mathrm{pr}_i$ it holds that
\begin{itemize}
\item $\mu_{ij} = \mu_{ji}$,
\item $\sum_{i,j=1}^n x_ix_j \mu_{ij}$ is a positive measure for all $x \in \mathbb{R}^n$
\item and 
$$\int_{\mathbb{R}^n} \mathrm{pr}_i(\mathbf{y}) \partial_j \phi(\mathbf{y}) \, \mathrm{d}\mathbf{y} = - \int_{\mathbb{R}^n} \phi(\mathbf{y})
\mu_{ij}(\mathrm{d}\mathbf{y})$$
for all $\phi \in C^{\infty}(\mathbb{R}^n)$ with 
\begin{equation} \label{eq:bound}
\sup_{\mathbf{y} \in \mathbb{R}^n} (1 + ||\mathbf{y}||_2^2)^N
\max\left\{|\phi(\mathbf{y})|, |\partial_1 \phi(\mathbf{y})|, \ldots, |\partial_n \phi(\mathbf{y})|\right\} < \infty
\end{equation}
for all $N \in \mathbb{N}_0$. 
\end{itemize}
\end{lemma}

\begin{proof} First recall that 
$$|\mathrm{pr}_i(\mathbf{y})| \leq ||\mathrm{pr}(\mathbf{y})||_2 \leq ||\mathrm{pr}(0)||_2 + ||\mathbf{y}||_2,$$
which proves that $\mathrm{pr}_i$ is in $L^1_{\text{loc}}$. A standard
mollifier argument gives that for all $x \in \mathbb{R}^n$ 
$$\phi \mapsto \int_{\mathbb{R}^n} \rho(\mathbf{y}) \sum_{i,j=1}^n x_i x_j \partial_i \partial_j
\phi(\mathbf{y}) \, \mathrm{d}\mathbf{y}$$
is a positive linear functional on $C_c^{\infty}(\mathbb{R}^n)$
due to convexity of $\rho$. Riesz's
representation theorem gives the existence of a Radon
measure $\mu^{x}$ such that 
$$\int_{\mathbb{R}^n} \rho(\mathbf{y}) \sum_{i,j=1}^n x_i x_j \partial_i \partial_j
\phi(\mathbf{y}) \, \mathrm{d}\mathbf{y} = \int_{\mathbb{R}^n} \phi(\mathbf{y}) \mu^{x}(\mathrm{d}\mathbf{y}).$$
Taking $\mu_{ii} = \mu^{e_i}$ and 
$$\mu_{ij} = \mu^{(e_i + e_j)/{\sqrt{2}}} - \mu_{ii} - \mu_{jj}$$
for $i \neq j$ gives the existence of signed Radon measures $\mu_{ij}$, which by construction
fulfill the two first bullet points. Since $\rho$ is convex, it is
locally Lipschitz continuous, hence weakly differentiable with first
weak partial derivatives coinciding with the pointwise partial derivatives, $\mathrm{pr}_i(\mathbf{y})$, for Lebesgue almost all $\mathbf{y}$. Hence 
$$\int_{\mathbb{R}^n} \mathrm{pr}_i(\mathbf{y}) \partial_j \phi(\mathbf{y}) \, \mathrm{d}\mathbf{y} = - \int_{\mathbb{R}^n}
\rho(\mathbf{y}) \partial_i \partial_j \phi(\mathbf{y}) \, \mathrm{d} \mathbf{y} = - \int_{\mathbb{R}^n} \phi(\mathbf{y})
\, \mu_{ij}(\mathrm{d}\mathbf{y})$$
for all $\phi \in C^{\infty}_c(\mathbb{R}^n)$. 
We then prove that the
partial integration formula generalizes to all $\phi \in
C^{\infty}(\mathbb{R}^n)$ that fulfill (\ref{eq:bound}). To this end 
fix a positive function $\kappa \in C_c^{\infty}(\mathbb{R}^n)$ such that $\kappa(\mathbf{y}) = 1$ for
$||\mathbf{y}||_2 \leq 1$ . Define 
$$q_r(\mathbf{y}) = (1 + ||\mathbf{y}||^2_2)^{-N} \kappa(r \mathbf{y}),$$ 
then $q_r \in C_c^{\infty}(\mathbb{R}^n)$ and 
$$q_r(\mathbf{y}) \geq (1 + ||\mathbf{y}||^2_2)^{-N} 1(r ||\mathbf{y}||_2 \leq 1) \to (1 + ||\mathbf{y}||^2_2)^{-N}$$
for $r \to 0$. By monotone convergence 
$$\int_{\mathbb{R}^n} q_r(\mathbf{y}) \, \mu_{ij}(\mathrm{d}\mathbf{y}) \to \int_{\mathbb{R}^n} (1 +
||\mathbf{y}||^2_2)^{-N} \, \mu_{ij}(\mathrm{d}\mathbf{y})$$
for $r \to 0$. Moreover, $\kappa(r\mathbf{y}) = 1$ and $\partial_j \kappa(r\mathbf{y}) =
0$ for $||\mathbf{y}||_2 \leq 1/r$, hence 
$$\partial_j q_r(\mathbf{y}) \rightarrow \partial_j (1 + ||\mathbf{y}||_2^2)^{-N}$$
for $r \to 0$. Since 
$$|\mathrm{pr}_i(\mathbf{y}) \partial_j q_r(\mathbf{y})| \leq p(\mathbf{y}) (1 +
||\mathbf{y}||^2_2)^{-2N}$$
for some polynomial $p(\mathbf{y})$ of degree $N + 1$ independent of $r$ (for
$r \leq 1$, say), and since the upper bound is integrable w.r.t. the
$n$-dimensional Lebesgue measure for $N$ large enough,  
it follows by dominated convergence that for 
$N$ large enough
$$ \int_{\mathbb{R}^n} (1 + ||\mathbf{y}||_2^2)^{-N} \, \mu_{ij}(\mathrm{d}\mathbf{y}) = - \int_{\mathbb{R}^n} \mathrm{pr}_i(\mathbf{y}) \partial_j
(1 + ||\mathbf{y}||_2^2)^{-N} \, \mathrm{d} \mathbf{y}.$$
The function $\mathbf{y} \mapsto (1 + ||\mathbf{y}||_2^2)^{-N}$ is, in particular, 
$\mu_{ij}$-integrable. If $\phi \in C^{\infty}(\mathbb{R}^n)$ fulfills
(\ref{eq:bound}) we let $\phi_r(\mathbf{y}) = \phi(\mathbf{y}) \kappa(r\mathbf{y})$. Then $\phi_r
\in C_c^{\infty}(\mathbb{R}^n)$, $\phi_r(\mathbf{y}) \to \phi(\mathbf{y})$ for $r \to 0$, and 
$$\partial_j \phi_r(\mathbf{y}) = \partial_j \phi(\mathbf{y}) \kappa(r\mathbf{y}) +
\phi(\mathbf{y}) r \partial_j \kappa(r\mathbf{y}) \rightarrow \partial_j \phi(\mathbf{y})$$
for $r \to 0$. Moreover, for $r \leq 1$ there is a constant $C_N$ such
that 
$$|\mathrm{pr}_i(\mathbf{y}) \partial_j \phi_r(\mathbf{y})| \leq C_N (1 + ||\mathbf{y}||_2^2)^{-N + 2}$$
as well as
$$|\phi_r(\mathbf{y})| \leq C_N (1 + ||\mathbf{y}||_2^2)^{-N}$$
since $\phi$ fulfills (\ref{eq:bound}). Again by Lebesgue as well
as $\mu_{ij}$-integrability of the upper bound for $N$ large enough, it
follows from dominated convergence that 
\begin{eqnarray*}
\int_{\mathbb{R}^n} \mathrm{pr}_i(\mathbf{y}) \partial_j \phi(\mathbf{y}) \, \mathrm{d}\mathbf{y} & = & \lim_{r \to
  0} \int_{\mathbb{R}^n} \mathrm{pr}_i(\mathbf{y}) \partial_j \phi_r(\mathbf{y}) \, \mathrm{d}\mathbf{y}\\ 
& = & -  \lim_{r \to 0}  \int_{\mathbb{R}^n} \phi_r(\mathbf{y}) \mu_{ij}(\mathrm{d}\mathbf{y})\\
& = & - \int_{\mathbb{R}^n} \phi(\mathbf{y}) \mu_{ij}(\mathrm{d}\mathbf{y}).
\end{eqnarray*}
\end{proof}

The first part of the proof of Lemma \ref{lem:bv}, where we establish
the existence of the $\mu_{ij}$-measures, follows the proof of
Theorem 6.3.2 in \cite{Evans:1992}. In the remaining part we 
effectively prove that $\mathrm{pr}_i$ is a tempered
distribution. This actually follows directly from the 
polynomial bound on $\mathrm{pr}_i$ by Example 7.12(c) in
\cite{Rudin:1991}. However, we need a little more than just the fact
that the continuous linear functional
$$\phi \mapsto \int_{\mathbb{R}^n} \mathrm{pr}_i(\mathbf{y}) \partial_j \phi(\mathbf{y}) \, \mathrm{d}
\mathbf{y}$$
on the test functions $C_c^{\infty}(\mathbb{R}^n)$ extends to a
continuous linear functional on the Schwartz space $\mathcal{S}$ of
rapidly decreasing functions. We also need the explicit form of the
extension (the partial integration formula) as stated in Lemma
\ref{lem:bv}. 

To finally prove Theorem \ref{thm:main} we need to relate the distributional partial
derivatives $\mu_{ij}$ of $\mathrm{pr}_i$ to the pointwise partial
derivatives $\partial_{j} \mathrm{pr}_i$ defined Lebesgue almost
everywhere. To this end we need the concept of 
approximate differentiability. 

\begin{dfn} Let $m_n$ denote the $n$-dimensional Lebesgue measure and $B(\mathbf{y}, r)$
the $\ell_2$-ball with center $\mathbf{y}$ and radius $r$. A function $f : \mathbb{R}^n \to \mathbb{R}^n$ is
  approximately differentiable in $\mathbf{y}$ if there is a matrix $A$ such
  that for all $\epsilon > 0$ 
$$\frac{1}{m_n(B(\mathbf{y}, r))}m_n\left(\left\{ \mathbf{x} \in B(\mathbf{y}, r) \;\middle|\; \frac{||f(\mathbf{x}) - f(\mathbf{y})
        - A(\mathbf{x}-\mathbf{y})||}{||\mathbf{x}-\mathbf{y}||_2} \geq \epsilon \right\}\right) \rightarrow
    0$$
for $r \to 0$. 
\end{dfn}

By Theorem 6.1.3 in \cite{Evans:1992} the matrix $A$ is unique if $f$
is approximately differentiable in $\mathbf{y}$. It is called the approximate
derivative of $f$ in $\mathbf{y}$. Note that approximate differentiability of
$f$ in $\mathbf{y}$ is a local property, which only requires that $f$ is
defined Lebesgue almost everywhere in a neighborhood of $\mathbf{y}$.

\begin{lemma} \label{lem:extapp} If $f : D \to \mathbb{R}^n$ is differentiable
  in $\mathbf{y}$ in the extended sense then $f$ is approximately
  differentiable in $\mathbf{y}$ with the same derivative. 
\end{lemma}

\begin{proof} Assume that $f$ is differentiable in $\mathbf{y}$ in the extended
  sense with derivative $A$. We can then for fixed $\epsilon > 0$ choose $r$ sufficiently small
  such that $D^c \cap B(\mathbf{y}, r)$ is a Lebesgue null set and 
$$\frac{||f(\mathbf{y}) - f(\mathbf{x}) - A(\mathbf{x} - \mathbf{y})||_2}{||\mathbf{x}- \mathbf{y}||_2} < \epsilon$$
for $\mathbf{x} \in D \cap B(\mathbf{y}, r)$. Choosing an arbitrary extension of
$f$ to $B(\mathbf{y}, r)$ we find that  
$$\left\{\mathbf{x} \in B(\mathbf{y},r)  \;\middle|\; \frac{||f(\mathbf{y}) - f(\mathbf{x}) - A(\mathbf{x} -
    \mathbf{y})||_2}{||\mathbf{x}- \mathbf{y}||_2} \geq \epsilon \right\} \subseteq D^c \cap B(\mathbf{y},r),$$
which implies that $f$ is approximately differentiable in $\mathbf{y}$ with
derivative $A$. 
\end{proof}

If $f: \mathbb{R}^n \to \mathbb{R}^n$ has coordinates of locally
bounded variation with corresponding distributional partial derivatives of $f_i$
denoted $\mu_{ij}$ for $j = 1, \ldots, n$ we have by Lebesgue's decomposition theorem that 
$$\mu_{ij} = h_{ij} \cdot m_n +  \nu_{ij}$$
with $\nu_{ij} \perp m_n$. We can now state (and subsequently use) a well known
but rather deep result on approximate differentiability of functions of locally bounded
variation. See Theorem 6.1.4 in \cite{Evans:1992}.  

\begin{theorem} \label{thm:app} If $f: \mathbb{R}^n \to \mathbb{R}^n$ has coordinates of locally
bounded variation then $f_i$ is approximately differentiable for
Lebesgue almost all $\mathbf{y}$ with derivative $(h_{i1}(\mathbf{y}), \ldots, h_{in}(\mathbf{y}))$.
\end{theorem}

It is straightforward to see that if $f$ has coordinates of locally
bounded variation then it is also, as a function from $
\mathbb{R}^n$ to $\mathbb{R}^n$, approximately differentiable
for Lebesgue almost all $\mathbf{y}$ with derivative $(h_{ij}(\mathbf{y}))_{i,j=1, \ldots, n}$.

\begin{proof}[Proof of Theorem \ref{thm:main}] From Lemma \ref{lem:bv},
$\mathrm{pr}_i$ is of locally bounded variation with distributional partial
derivatives $\mu_{ij}$. Combining Theorem \ref{thm:prdif}, Lemma
\ref{lem:extapp} and Theorem \ref{thm:app} -- and using that the
approximate derivative is unique -- we conclude that 
$$\mu_{ij} = \partial_j \mathrm{pr}_{i} \cdot m_n +  \nu_{ij}$$
with  $\nu_{ij} \perp m_n$. 

Letting $\psi(\mathbf{y}; \boldsymbol{\xi}, \sigma^2)$ denote the density for the multivariate normal
distribution with mean $\boldsymbol{\xi}$ and covariance matrix $\sigma^2 I$ we have that 
$$\partial_j \psi(\mathbf{y}; \boldsymbol{\xi}, \sigma^2) = - \frac{(y_j - \xi_j)}{\sigma^2}
\psi(\mathbf{y}; \boldsymbol{\xi}, \sigma^2).$$
Since $\psi(\cdot; \boldsymbol{\xi}, \sigma^2) \in C^{\infty}(\mathbb{R}^n)$ fulfills (\ref{eq:bound}), Lemma
\ref{lem:bv} implies that 
{\small
\begin{eqnarray*}
\mathrm{cov}(Y_i, \mathrm{pr}_i(\mathbf{Y})) & = & \int_{\mathbb{R}^n} \mathrm{pr}_i(\mathbf{y}) (y_i - \xi_i) \psi(\mathbf{y}; \boldsymbol{\xi},
\sigma^2) \, \mathrm{d} \mathbf{y} \\
& = & - \sigma^2 \int_{\mathbb{R}^n} \mathrm{pr}_i(\mathbf{y}) \partial_i \psi(\mathbf{y}; \boldsymbol{\xi},
\sigma^2) \, \mathrm{d} \mathbf{y} \\
& = & \sigma^2 \int_{\mathbb{R}^n} \psi(\mathbf{y}; \boldsymbol{\xi}, \sigma^2) \, \, \mu_{ii}(\mathrm{d}\mathbf{y}) \\
& = & \sigma^2 \int_{\mathbb{R}^n} \psi(\mathbf{y}; \boldsymbol{\xi}, \sigma^2) \partial_i
\mathrm{pr}_{i}(\mathbf{y}) \, \mathrm{d} \mathbf{y} + \sigma^2 \int_{\mathbb{R}^n} \psi(\mathbf{y}; \boldsymbol{\xi}, \sigma^2) \,
\nu_{ii}(\mathrm{d}\mathbf{y}).
\end{eqnarray*}}

Theorem (\ref{thm:main}) follows by division with $\sigma^2$ and
summation over $i$, which gives that 
$$\nu = \sum_{i=1}^n \nu_{ii}.$$
\end{proof}

\begin{proof}[Proof of Proposition \ref{prop:supp}] The set
  $U = \mathbb{R}^n \backslash \overline{\mathrm{exo}(K)}$ is open.
  If $\mathrm{pr}_i$ is locally Lipschitz on $U$ Theorem
  4.2.5 in \cite{Evans:1992} 
  gives that $\mathrm{pr}_i$ is weakly differentiable,
 and the weak partial derivative  in the  $j$'th
  direction coincides with the 
  Lebesgue almost everywhere defined $\partial_j \mathrm{pr}_i$. That
  is, 
$$\int_{\mathbb{R}^n} \mathrm{pr}_i(\mathbf{y}) \partial_j \phi(\mathbf{y}) \, \mathrm{d} \mathbf{y} = - \int_{\mathbb{R}^n} \partial_j
\mathrm{pr}_i(\mathbf{y}) \phi(\mathbf{y}) \, \mathrm{d}\mathbf{y}$$
for all $\phi \in C_c^{\infty}(\mathrm{R}^n)$. It follows that 
$$\mu_{ij} = \partial_j \mathrm{pr}_i \cdot m_n,$$
and all the singular measures $\nu_{ij}$ are null measures.
\end{proof}

\section{Discussion} \label{sec:dis}

Our main result obtained in this paper is Theorem \ref{thm:main}. 
It characterises the size of $\mathrm{df} - \mathrm{df}_S$, which can
be interpreted as how much the non-convexity of $K$ affects the
degrees of freedom. From Theorem \ref{thm:main} we observe that
$\mathrm{df} - \mathrm{df}_S \geq 0$, and its magnitude is determined by how 
large $\psi(\mathbf{y}; \boldsymbol{\xi}, \sigma^2)$ is on the Lebesgue null set $N$ where 
the singular measure $\nu$ is concentrated. This is, in turn, determined
by the distance (scaled by $1/\sigma$) from $\boldsymbol{\xi}$ to points in
$N$ in combination with the distribution of the mass of the
measure $\nu$ on $N$. The singular measure depends only on $K$, and it
represents global geometric properties of $K$. 

Previous results on degrees of freedom in \cite{Tibshirani:2012},
\cite{Kato:2009}, \cite{Zou:2007} and \cite{Meyer:2000} all correspond
to $K$ being convex, and the resulting Lipschitz continuity of the
metric projection implies pointwise differentiability almost
everywhere by Rademacher's Theorem. To establish pointwise
differentiability almost everywhere of the metric projection onto any
closed set, we relied instead on the fact that it is the derivative
of a convex function. We then used Alexandrov's theorem for convex
functions to establish almost everywhere differentiability of the
metric projection. This is in principle well known in the mathematical
literature, and Asplund provided, for instance, only a brief argument
in \cite{Asplund:1973} for what is close to being Theorem
\ref{thm:prdif}. However, we needed to clarify in what sense the
metric projection is differentiable, and the precise relationship
between pointwise derivatives Lebesgue almost everywhere and
distributional derivatives for which partial integration applies. The
original formulation of Alexandrov's theorem was, in particular,
stated as the existence of a quadratic expansion of a convex funktion
$g$ for Lebesgue almost all $y$. This formulation does not require a
definition of differentiability of $\nabla g$ in $y$ in cases where
$\nabla g$ is not defined in a neighborhood of $y$. Consequently, the
conclusion cannot be formulated in terms of $\nabla g$ alone. The more
recent formulation of Alexandrov's theorem as in Theorem
\ref{thm:alexandrov} was useful, since it allowed us to formulate
Theorem \ref{thm:prdif} in terms of differentiability properties of
the metric projection itself rather than as a quadratic expansion of
$\rho$.

We gave three simple examples where analytic computations could shed
some light on the general results, and then we considered a more
serious application in Section \ref{sec:ode} on the estimation of
parameters in a $d$-dimensional linear ODE. This example served
several purposes. First we used it to test our algorithms for
computing the nonlinear $\ell_1$-regularized least squares estimator,
and we used it to test the divergence formula given in Theorem
\ref{thm:divlasso}. For the chosen model and parameter set and the
$\ell_1$-constrained estimator we concluded that
$\widehat{\mathrm{Risk}}(s)$ was, for all practical purposes,
unbiased, that it was useful for selection of $s$, and that the
selected model had a lower risk than e.g. the MLE.  The example also
showed that in this case the approximation $|\mathcal{A}| - 1$ to the
divergence was sufficiently accurate to be a computationally cheap
alternative to the formula from Theorem \ref{thm:divlasso}. When we
considered model searching instead, the risk estimate based on the
divergence became biased, and tended to select too complex models. Our
conclusion is that for the $\ell_1$-constrained estimator, the set $K$
may be non-convex, but this presents no problem for the estimation of
the degrees of freedom by the divergence. On the contrary, when 
$K$ is a union of models and we perform model searching, the
non-convexity of $K$ implies that the 
divergence underestimates the degrees of freedom considerably.

For practical applications we are faced with three challenges: We need
to compute the divergence to estimate $\mathrm{df}_S$; we need to
control, estimate or bound the difference $\mathrm{df} -
\mathrm{df}_S$; and we need to know or estimate $\sigma^2$. For the
latter, the typical solution is to estimate $\sigma^2$ in
an (approximately) unbiased way. In 
the ODE example an estimate of $\sigma^2$ can be based on the MLE
of $B$. For the computation of the divergence
we gave two formulas for parametrized models. We expect that similar formulas can be derived 
via implicit differentiation in cases where we have a parametrized
model, but with different restrictions on the parameters than we considered.
Alternatively, abstract results in Chapter 13 in
\cite{Rockafellar:1998} can be considered. The greatest challenge is
to control the difference $\mathrm{df} - \mathrm{df}_S$. Our
simulations showed an example where this difference was negligible 
as well as an example where it was not. In some simple cases we
were also able to compute the measure $\nu$, which can be used to compute the
difference. We do not expect that it will be an easy task to compute $\nu$
in many cases of practical interest, but we do expect that it will be
possible for best subset selection. We also expect that it will be
possible to make analytic progress on the further characterization of
$\nu$ and its support, e.g. when it has a density w.r.t. the
$(n-1)$-dimensional Hausdorff measure. It is, in addition, possible
to show that under a so-called prox-regularity assumption on
$\mathrm{pr}(\mathbf{y})$, the metric projection is Lipschitz in a 
neighborhood of $\mathrm{pr}(\mathbf{y})$, see \cite{Poliquin:2000}. 
Thus in this neighborhood $\nu$ is 0. This can be a path for bounding 
$\mathrm{df} - \mathrm{df}_S$ if $\boldsymbol{\xi}$ is close to
$K$. Even if it appears to be a challenging path, our representation of  
$\mathrm{df} - \mathrm{df}_S$ in terms of the singular measure $\nu$
does provide us with a novel way to achieve further progress.

\bibliographystyle{agsm}
\bibliography{../texbib-1}

\end{document}

\newtheorem{theorem}{Theorem}
\newtheorem{lemma}{Lemma}
\newtheorem{corollary}{Corollary}
\newtheorem{proposition}{Proposition}
\theoremstyle{definition}
\newtheorem{dfn}{Definition}
\newtheorem{ex}{Example}
\renewcommand{\phi}{\varphi}
\renewcommand{\epsilon}{\varepsilon}

\title[Divergence formulas and Algorithms]{Supplementary material: \\ Divergence formulas and Algorithms}
\author[N. R. Hansen]{Niels Richard Hansen} 
\address{Department of Mathematical Sciences,
University of Copenhagen,
Universitetsparken 5,
2100 Copenhagen \O,
Denmark}
\email[Corresponding author]{Niels.R.Hansen@math.ku.dk}

\author[A. Sokol]{Alexander Sokol}
\email{alexander@math.ku.dk}

\begin{abstract}

  This is supplementary material for the paper \emph{Degrees of
    freedom for nonlinear least squares estimation}. It contains the
  derivation of the divergence formulas, additional details related to
  the other proofs, technical details on the algorithms and
  implementations, and some additional simulation results.

\end{abstract}

\maketitle

\noindent

\section{Properties of the function $\rho$} In this section we give
the central but well known result that the metric projection onto
a closed set can be expressed as a subdifferential of a convex function. 

\begin{lemma} \label{lem:1} Assume that $K \subseteq \mathbb{R}^n$ is a nonempty
  and closed set. The function 
$$\rho(\mathbf{y}) = \sup_{\mathbf{x} \in K} \{ \mathbf{y}^T \mathbf{x} - ||\mathbf{x}||^2/2\}$$
is convex. With $\partial \rho$ denoting the subdifferential of $\rho$
then $\partial \rho(\mathbf{y})$ contains the set of points in $K$ closest to
$\mathbf{y}$. If $\rho$ is differentiable in $\mathbf{\mathbf{y}}$ with gradient $\nabla
\rho (\mathbf{y})$, then the metric projection of $\mathbf{y}$ onto
$K$ is unique, and $\mathrm{pr}(\mathbf{y}) = \nabla \rho (\mathbf{y})$.
\end{lemma}

\begin{proof} Since $\rho$ is the pointwise supremum of the affine (thus convex) functions
$$\mathbf{y} \mapsto  \mathbf{y}^T \mathbf{x} - ||\mathbf{x}||^2/2 = ||\mathbf{y}||^2/2 - ||\mathbf{y} - \mathbf{x}||^2/2,$$
it is convex, and 
$$\rho(\mathbf{y}) = ||\mathbf{y}||^2/2 - \inf_{\mathbf{x} \in K}  ||\mathbf{y} - \mathbf{x}||^2/2.$$
With 
$$\mathrm{Pr}(\mathbf{y}) = \argmin_{\mathbf{x} \in K}  ||\mathbf{y} - \mathbf{x}||^2$$ 
the nonempty set of points in $K$ closest to $\mathbf{y}$ it follows that 
$$\rho(\mathbf{y}) = \mathbf{y}^T \mathbf{x}  - ||\mathbf{x}||^2/2$$
for all $\mathbf{x} \in \mathrm{Pr}(\mathbf{y})$. For $\mathbf{x} \in
\mathrm{Pr}(\mathbf{y})$
\begin{eqnarray*}
\rho(\mathbf{y} + \mathbf{z}) & = & \sup_{\mathbf{x} \in K} \{
\mathbf{y}^T \mathbf{x} - ||\mathbf{x}||^2/2 + \mathbf{z}^T
\mathbf{x}\} \geq \mathbf{y}^T \mathbf{x}  - ||\mathbf{x}||^2/2 + 
\mathbf{z}^T\mathbf{x} \\
& = & \rho(\mathbf{y}) + \mathbf{z}^T \mathbf{x},
\end{eqnarray*}
which shows that $\mathrm{Pr}(\mathbf{y}) \subseteq \partial \rho(\mathbf{y})$ by
definition of the subdifferential. If $\rho$ is differentiable,
$$\partial \rho(\mathbf{y}) = \{ \mathrm{pr}(\mathbf{y}) \} = \mathrm{Pr}(\mathbf{y}),$$
and the last claim follows. 
\end{proof}

Let $D \subseteq \mathbb{R}^n$ denote the domain of $\nabla \rho$
on which $\rho$ is differentiable. The following observation is
useful. If $\mathbf{y} \in D$, if $\mathbf{y}_n \to \mathbf{y}$ 
and if $\mathbf{z}_n \in \mathrm{Pr}(\mathbf{y}_n)$ converges to $\mathbf{z}$ then 
$$  \rho(\mathbf{y} + \mathbf{x}) = \lim_{n \to \infty}  \rho(\mathbf{y}_n + \mathbf{x}) \geq \lim_{n \to
  \infty} \rho(\mathbf{y}_n) + \mathbf{x}^T \mathbf{z}_n \geq \rho(\mathbf{y}) + \mathbf{x}^T \mathbf{z},$$ 
which implies that $\mathbf{z} \in \partial \rho (\mathbf{y}) = \{\mathrm{pr}(\mathbf{y})\}$,
whence $\mathbf{z} = \mathrm{pr}(\mathbf{y})$. This proves a continuity property of the
metric projection: If $\mathbf{y} \in D$ and $U$ is a neighborhood of
$\mathrm{pr}(\mathbf{y})$ then $\{ \mathbf{z} \in \mathbb{R}^n \mid
\mathrm{Pr}(\mathbf{z}) \subseteq U \}$ contains a neighborhood of
$\mathbf{y}$. We will need this continuity property when deriving the
divergence formulas below.

\section{Proofs of the divergence formulas} The
formulas for computation of the divergence given in Section
3 of the paper will be proved using the implicit function
theorem to compute the divergence of $\zeta(\hat{\beta})$.  
To connect such a local result expressed in the $\beta$ parametrization
to the divergence of the globally defined metric projection
we will first establish that there is a
neighborhood of $\mathbf{y}$ where the (global) metric projection can be found
by minimizing $||\mathbf{z} - \zeta(\beta)||^2_2$ in a neighborhood of
$\hat{\beta}$. Note that $\mathrm{Pr}(\mathbf{z})$ denotes, as in the proof of
Lemma \ref{lem:1}, the set of metric projections of $\mathbf{z}$.  

\begin{lemma} \label{lem:local} If the regularity assumptions on $\zeta$ as stated in
  Section 3 hold, then for all neighborhoods $V$ of
  $\hat{\beta}$ there exists a neighborhood $N$ of $\mathbf{y}$ such that  
$$\mathrm{Pr}(\mathbf{z}) = \zeta(\argmin_{\beta \in V \cap \Theta} ||\mathbf{z} - \zeta(\beta)||^2_2)$$
for $\mathbf{z} \in N$.
\end{lemma}

\begin{proof} With $V$ a neighborhood of $\hat{\beta}$ there is, since
  $\zeta$ was assumed to be open at $\hat{\beta}$, a
  neighborhood $U$ of $\mathrm{pr}(\mathbf{y}) = \zeta(\hat{\beta})$ such that 
$$U \cap K \subseteq \zeta(V \cap \Theta).$$
By the continuity property of the metric projection there is a
neighborhood $N$ of $\mathbf{y}$ such that $\mathrm{Pr}(\mathbf{z}) \subseteq U$ for $\mathbf{z} \in
N$. By definition, $\mathrm{Pr}(\mathbf{z}) \subseteq K$, hence 
$$\mathrm{Pr}(\mathbf{z}) \subseteq \zeta(V \cap \Theta).$$
This proves first that $W = \argmin_{\beta \in V \cap \Theta} ||\mathbf{z} -
\zeta(\beta)||^2_2$ is not empty, and second that $\beta \in
W$ if and only if $\zeta(\beta) \in \mathrm{Pr}(\mathbf{z})$. 
\end{proof}

Below we use the implicit function
  theorem to show that for neighborhoods $N$ of $\mathbf{y}$ and $V$ of
  $\hat{\beta}$ there exists a $C^1$-map $\hat{\beta} : N \to V \cap \Theta$ such
  that $\zeta \circ \hat{\beta} : N \to K$ satisfies 
$$\{\zeta \circ \hat{\beta}(\mathbf{z})\} = \zeta(\argmin_{x \in V \cap \Theta}
||\mathbf{z} - \zeta(\beta)||^2_2).$$
It follows from Lemma \ref{lem:local} above that 
$$\mathrm{pr}(\mathbf{z}) = \zeta \circ \hat{\beta}(\mathbf{z})$$
for $\mathbf{z}$ in a neighborhood (contained in $N$) of $\mathbf{y}$.  
This ensures that 
\begin{equation} \label{eq:divid}
\nabla \cdot \mathrm{pr}(\mathbf{y}) = \nabla \cdot \zeta \circ \hat{\beta}(\mathbf{y}).
\end{equation}

Now recall the definitions of the $G$ and $J$ matrices,
\begin{equation} \label{eq:Gdfn}
G_{kl} = \sum_{i=1}^n \partial_k \zeta_i(\hat{\beta}) \partial_l
\zeta_i(\hat{\beta})
\end{equation}
and 
\begin{equation} \label{eq:Jdfn}
J_{kl} = G_{kl} - \sum_{i=1}^n (y_i -
\zeta_i(\hat{\beta})) \partial_k \partial_l \zeta_i(\hat{\beta}).
\end{equation}
The next lemma on differentiation of the
quadratic loss is a straightforward computation, and its proof is 
left out.

\begin{lemma} \label{lem:help} If $\zeta$ is $C^2$ in a neighborhood of
  $\beta$ then $f(\mathbf{z}, \beta) = \frac{1}{2} ||\mathbf{z} - \zeta(\beta)||_2^2$ is
  $C^2$ in a neighborhood of $(\mathbf{y}, \beta)$ with 
$$\partial_{z_i} \partial_k f(\mathbf{z}, \beta) = - \partial_k \zeta(\beta)$$
and 
$$\partial_k \partial_l f(\mathbf{z}, \beta) = J_{kl},$$
where $J_{kl}$ is given by (\ref{eq:Jdfn}). 
\end{lemma}

Note that in the notation above, $\partial_k$ refers to differentiation
w.r.t. to $\beta_k$ and $\partial_{z_i}$ refers to differentiation
w.r.t. $z_i$. 

\begin{proof}[Proof of  Theorem 3] With $f$ as in
  Lemma \ref{lem:help} the estimator $\hat{\beta}$ fulfills
$$\nabla_{\beta} f(\mathbf{y}, \hat{\beta}) = 0,$$
with the Jacobian of the map $\beta \mapsto \nabla_{\beta} f(\mathbf{y}, \beta) $ being $J$
by Lemma \ref{lem:help}. Since $J$ has full rank by assumption the
implicit function theorem implies that there is a continuously differentiable solution map
$\hat{\beta}(\mathbf{z})$, defined in a neighborhood of $\mathbf{y}$, such that 
$$\nabla_{\beta} f(\mathbf{z}, \hat{\beta}(\mathbf{z})) = 0.$$
Moreover, $D_\mathbf{z} \nabla_{\beta} f(\mathbf{y}, \hat{\beta}) = - D_{\beta}
\zeta(\hat{\beta})^T$ by Lemma \ref{lem:help}, which gives by implicit
differentiation that 
$$D_\mathbf{z} \hat{\beta}(\mathbf{y}) = J^{-1} D\zeta(\hat{\beta})^T.$$
Hence,
$$D_\mathbf{z} (\zeta \circ \hat{\beta})(\mathbf{y}) =  D\zeta(\hat{\beta}) J^{-1} D\zeta(\hat{\beta})^T.$$
It follows from (\ref{eq:divid}) that 
$$\nabla \cdot \mathrm{pr}(\mathbf{y}) = \mathrm{tr}(D\zeta(\hat{\beta}) J^{-1}
D\zeta(\hat{\beta})^T) = \mathrm{tr}(J^{-1} D\zeta(\hat{\beta})^T D\zeta(\hat{\beta})) = \mathrm{tr}(J^{-1} G),$$
since $G =  D\zeta(\hat{\beta})^T D\zeta(\hat{\beta})$ as defined by (\ref{eq:Gdfn}).
\end{proof}

\begin{proof}[Proof of  Theorem 4]  With $f$ as in
  Lemma \ref{lem:help} the estimator $\hat{\beta}$ fulfills, by
  assumption, 
$$\nabla_{\beta} f(\mathbf{y}, \hat{\beta}) =  \hat{\lambda} \gamma$$
for $\hat{\lambda} > 0$, $\gamma \in \mathbb{R}^p$, $\gamma_k = \omega_k
\mathrm{sign}(\hat{\beta}_k)$ if $\hat{\beta}_k \neq 0$ and 
$\gamma_k \in (-\omega_k,\omega_k)$ if $\hat{\beta}_k = 0$. Moreover,
as $\hat{\lambda} > 0$ it holds that $\sum_{k=1}^p \gamma_k \beta_k =
s.$ In the following we identify any $\mathbb{R}^{\mathcal{A}}$-vector 
denoted $\beta_{\mathcal{A}}$ with an $\mathbb{R}^p$ vector with 0's
in entries with indices not in $\mathcal{A}$. We introduce the map 
$$R(\mathbf{z}, \beta_{\mathcal{A}}, \lambda) = \left(\begin{array}{c} 
\nabla_{\beta_{\mathcal{A}}} f(\mathbf{z}, \beta_{\mathcal{A}}) - \lambda \gamma_{\mathcal{A}} \\
\sum_{i=1}^p \gamma_k \beta_{\mathcal{A},k} - s
\end{array} \right),
$$ 
and we observe that $R(\mathbf{y}, \hat{\beta}_{\mathcal{A}}, \hat{\lambda}) =
0$. The derivative of $R$ is found to be  
$$D_{\beta_{\mathcal{A}}, \lambda} R(\mathbf{y}, \hat{\beta}_{\mathcal{A}}, \hat{\lambda})  
 = \left(\begin{array}{cc} 
J_{\mathcal{A}, \mathcal{A}} &  \gamma_{\mathcal{A}} \\
\gamma_{\mathcal{A}}^T & 0 
\end{array} \right).$$
By the assumptions made on $J_{\mathcal{A}, \mathcal{A}}$ this matrix
is invertible with 
$$\left(\begin{array}{cc} 
J_{\mathcal{A}, \mathcal{A}} &  \gamma_{\mathcal{A}} \\
\gamma_{\mathcal{A}}^T & 0 
\end{array} \right)^{-1} = \left(\begin{array}{cc} 
(J_{\mathcal{A}, \mathcal{A}})^{-1} - \frac{(J_{\mathcal{A}, \mathcal{A}})^{-1}
\gamma_{\mathcal{A}} \gamma_{\mathcal{A}}^T (J_{\mathcal{A}, \mathcal{A}})^{-1}}{\gamma_{\mathcal{A}}^T (J_{\mathcal{A}, \mathcal{A}})^{-1} \gamma_{\mathcal{A}}}  & *\\
* & *
\end{array} \right).$$
It follows from the implicit function theorem that there is a
neighborhood of $\mathbf{y}$ in which there is a continuously differentiable
solution map $(\hat{\beta}_{\mathcal{A}}(\mathbf{z}), \hat{\lambda}(\mathbf{z}))$ that fulfills
$R(\mathbf{z}, \hat{\beta}_{\mathcal{A}}(\mathbf{z}), \hat{\lambda}(\mathbf{z})) = 0.$
By the $C^2$-assumption the solution map fulfills the second order
sufficient conditions in a neighborhood of $\mathbf{y}$, and $\hat{\beta}_{\mathcal{A}}(\mathbf{z})$ is a local 
solution to the constrained optimization problem. Since $D_\mathbf{z} \nabla_{\beta} f(\mathbf{y}, \hat{\beta}) = - D_{\beta}
\zeta(\hat{\beta})^T$ by Lemma \ref{lem:help}, we get by implicit
differentiation that 
$$D_\mathbf{z} \hat{\beta}_{\mathcal{A}} (\mathbf{y}) = \left((J_{\mathcal{A}, \mathcal{A}})^{-1} - \frac{(J_{\mathcal{A}, \mathcal{A}})^{-1}
\gamma_{\mathcal{A}} \gamma_{\mathcal{A}}^T (J_{\mathcal{A}, \mathcal{A}})^{-1}}{\gamma_{\mathcal{A}}^T
(J_{\mathcal{A}, \mathcal{A}})^{-1} \gamma_{\mathcal{A}}}\right)
(D\zeta(\hat{\beta})_{\cdot, \mathcal{A}})^T.$$
Since $(D\zeta(\hat{\beta})_{\cdot, \mathcal{A}})^T
D\zeta(\hat{\beta})_{\cdot,\mathcal{A}} = G_{\mathcal{A},\mathcal{A}}$ it follows as in the proof of Theorem
3 that 
\begin{eqnarray*}
\nabla \cdot \mathrm{pr}(\mathbf{y}) & = & \mathrm{tr}\left( (J_{\mathcal{A},
    \mathcal{A}})^{-1} G_{\mathcal{A},\mathcal{A}} - \frac{(J_{\mathcal{A}, \mathcal{A}})^{-1}
\gamma_{\mathcal{A}} \gamma_{\mathcal{A}}^T (J_{\mathcal{A}, \mathcal{A}})^{-1} G_{\mathcal{A},\mathcal{A}}}{\gamma_{\mathcal{A}}^T
(J_{\mathcal{A}, \mathcal{A}})^{-1} \gamma_{\mathcal{A}}} \right) \\
& = & \mathrm{tr}\left( (J_{\mathcal{A},
    \mathcal{A}})^{-1} G_{\mathcal{A},\mathcal{A}} \right)- 
\frac{\gamma_{\mathcal{A}}^T (J_{\mathcal{A}, \mathcal{A}})^{-1}
  G_{\mathcal{A},\mathcal{A}} (J_{\mathcal{A}, \mathcal{A}})^{-1}
\gamma_{\mathcal{A}}}{\gamma_{\mathcal{A}}^T (J_{\mathcal{A}, \mathcal{A}})^{-1} \gamma_{\mathcal{A}}}.
\end{eqnarray*}
\end{proof}

\subsection{Summary of previous results in the mathematical literature} 
There is an extensive mathematical literature on the uniqueness, and to
some extent differentiability, of the metric projection -- in
particular in the infinite dimensional context. Some of these results
are related to our derivations of the divergence formulas above. 
\cite{Haraux:1977} showed results on the directional differentiability of the
metric projection onto a closed convex set in a Hilbert space. He
showed, in particular, that in finite dimensions the projection onto 
a polytope is directionally differentiable in $\mathbf{y}$ for all $\mathbf{y}$ with the 
directional derivative being the projection onto 
$$(\mathbf{y} - \mathrm{pr}(\mathbf{y}))^{\perp} \cap T_{\mathrm{pr}(\mathbf{y})}$$
where $T_{\mathrm{pr}(\mathbf{y})}$ is the tangent cone, see \cite{Haraux:1977}
for the details. This is a derivative
if and only if it is linear, which happens if and only if
$\mathrm{pr}(\mathbf{y})$ is in the relative interior of the face $(\mathbf{y} -
\mathrm{pr}(\mathbf{y}))^{\perp} \cap K$. This is also the face of smallest
dimension containing $\mathrm{pr}(\mathbf{y})$.  If we
consider an $\ell_1$-ball with radius $s$, and the solution is unique
with $p(s)$ nonzero parameters, the corresponding face has dimension
$p(s) - 1$. This result was also found in  \cite{Kato:2009}. 

\cite{Haraux:1977} showed, in addition, in his Example 2 how to compute the
derivative when the boundary of the set is $C^2$. The derivative is a
form of regularized projection onto the tangent plane at
$\mathrm{pr}(\mathbf{y})$ -- the regularization being determined by the
curvatures. Recently, \cite{Kato:2009} derived similar results in the
context of shrinkage estimation. Abatzoglou derived results in
\cite{Abatzoglou:1978}, but without assuming convexity. These previous
results are all closely related to our Theorem 3, but we chose to
downplay the differential geometric content. Instead, we discussed in
the paper its relation to TIC.

More recent results on differentiability of the metric projection can be found
in \cite{Rockafellar:1998}. Their Corollary 13.43 gives an abstract
result for a specific point, $\mathbf{y}$, where $\mathrm{pr}(\mathbf{y})$ is prox-regular w.r.t. $\mathbf{y}
- \mathrm{pr}(\mathbf{y})$, and the result applies, in particular, when $K$ is
fully amenable (regular enough). The result by Haraux on projections onto
polytopes follows from this general result --  see Example 13.44 in \cite{Rockafellar:1998}.

\section{Algorithms and Implementations}

The general implementation that computes $\ell_1$-penalized nonlinear
least squares estimates, as well as the implementation
of computations specifically related to linear ODEs are available in the R package
\texttt{smde}.  See
\url{http://www.math.ku.dk/~richard/smde/} for information on obtaining the R
package and the R code used for the results reported in Section 4 in
the paper.

In the following sections we describe some of the technical results
behind our implementation. In particular, the computation of
derivatives related to the matrix exponential.

\subsection{Differentiation of the matrix
  exponential} \label{sec:diffexp}
The map $A \to e^A$ is well known to be $C^{\infty}$ as a map from
$\mathbb{M}(d,d)$ to $\mathbb{M}(d,d)$. Moreover, its first and second
partial derivatives can be efficiently computed. We summarize a few
useful results from the literature. 

We denote by $L(A, F)$ the directional derivative of the matrix exponential in $A \in \mathbb{M}(d,d)$ in the general
direction $F \in \mathbb{M}(d,d)$. It has the analytic integral representation
\begin{equation} \label{eq:frechet}
L(A, F) = \int_0^1 e^{(1-u)A} F e^{uA} \, \mathrm{d} u.
\end{equation}
See e.g. (10.15) in \cite{Higham:2008}. If we use $\partial_{kl}$ to
denote the partial derivative w.r.t. the $(k,l)$'th entry, and if 
$E_{kl}$ denotes the $(k,l)$'th unit matrix, we have $\partial_{kl}
e^A = L(A, E_{kl})$. This gives the identity
\begin{equation} \label{eq:directional}
\mathrm{tr}(\partial_{kl} e^A  M) = 
\mathrm{tr} \left( E_{kl} \int_0^1  e^{u A} M e^{(1-u) A} \, \mathrm{d}
  u \right) = L(A, M)_{l,k}.
\end{equation}
for any $M \in \mathbb{M}(d,d)$. We will use this formula in the
following section. Efficient algorithms exist for
computing $L(A, F)$ for general matrices. It holds, for instance, that 
$$\exp\left( \left[\begin{array}{cc} 
A & F \\
0 & A 
\end{array} \right] \right)
= \left[ \begin{array}{cc}
e^A & L(A, F) \\ 
0 & e^A 
\end{array} \right],$$
see (10.43) in \cite{Higham:2008}, so if we can efficiently compute
matrix exponentials, we can compute the derivative. The
\texttt{expmFrechet} function in the \texttt{expm} R package,
\cite{expm:2012}, implements a faster algorithm that avoids the
dimension doubling.  

For the second partial derivatives it follows from (\ref{eq:frechet}) that 
$$\partial_{hr} \partial_{kl} e^{A} = H(A, E_{hr}, E_{kl}) + H(A, E_{kl}, E_{hr}),$$
where 
$$ H(A, F, G) = \int_0^1 \int_0^u e^{(1-u) A} F e^{(u - s) A} G e^{s A}
\, \mathrm{d} s \mathrm{d} u.$$
The computation of these iterated integrals is based on Theorem 1
in \cite{loan:1978}, which implies that
{ \small
$$\exp\left( \left[\begin{array}{cccc} 
A & F & 0 \\ 
0 & A  & G  \\
0 & 0 & A \\
\end{array} \right] \right) = \left[ \begin{array}{cccc}
e^A & L(A, F) & H(A, F, G)  \\ 
0 & e^A & L(A, G) \\
0 & 0 & e^A \\ 
\end{array} \right].$$
}
From the integral representation of $H(A, F, G)$ we find that for $M \in \mathbb{M}(d,d)$
\begin{eqnarray} \nonumber 
\mathrm{tr}(\partial_{hr} \partial_{kl} e^A  M) & = &  \mathrm{tr}(E_{hr}
H(A, E_{kl}, M)) +  \mathrm{tr}(E_{kl} H(A, E_{hr}, M)) \\
& = & H(A, E_{kl}, M)_{r,h} +  H(A, E_{hr}, M)_{l,k}, 
\label{eq:seconddirectional}
\end{eqnarray}
which was used for the computation of the $J$ matrix that enters in the
formula in Theorem 4.  

\subsection{Coordinate descent algorithm and sufficient
  transformations}
To solve the optimization problem 
$$\min_{\beta} ||y - \zeta(\beta)||_2^2 + \lambda \sum_{k=1}^p
\omega_k |\beta_k|$$
for a decreasing sequence of $\lambda$'s we have implemented a plain coordinate wise descent algorithm based on a standard
Gauss-Newton-type quadratic approximation of the loss
function. That is, for given $\beta
\in \Theta$ we approximate the loss in the $k$'th direction as
\begin{eqnarray*}
||y - \zeta(\beta + \delta e_k)||_2^2 & \simeq & ||r(\beta) - \partial_k
\zeta(\beta) \delta||^2_2 \\
& = & ||r(\beta)||_2^2 - 2 \langle r(\beta),  \partial_k
\zeta(\beta) \rangle \delta + || \partial_k \zeta(\beta)||_2^2 \delta^2
\end{eqnarray*}
where $r(\beta) = y - \zeta(\beta)$. The coordinate wise penalized
quadratic optimization problem can be solved explicitly, and we then
iterate over the coordinates until convergence. We implemented two
versions of the algorithm. Algorithm I is a generic algorithm that
relies on two auxiliary functions for computing $\zeta(\beta)$ and
$D\zeta(\beta)$. Algorithm II is specific to linear ODE models. With $m$ observations solving a
$d$-dimensional linear ODE, the computation time for Algorithm I scales
linearly with $m$, but the computation of $e^{tB} x$ and $D e^{tB} x$ can be
implemented to take advantage of sparseness of $B$. 
Algorithm II relies, on the other hand, on the precomputation of three
sufficient statistics, being $d \times d$ matrices, as outlined
below. For dense matrices the current implementation of Algorithm II 
scales better with $d$, and after the precomputation of the sufficient statistics, all other computation
times are independent of $m$. However, Algorithm II cannot take the same 
advantage of a sparse $B$. 

Since the loss is generally not convex, the steps may not be descent
steps if the quadratic approximation is poor. We implemented  Armijo
backtracking as described in \cite{Tseng:2009} to ensure sufficient decrease and
hence convergence. 

As mentioned above, Algorithm II for the linear ODE
example relies on sufficient statistics for the computation of the
loss as well as the quadratic approximation. We give here a brief
derivation of the necessary formulas. 
On $\mathbb{M}(d,d)$ the inner product can be expressed in terms of the trace,
$$\langle A, B \rangle = \mathrm{tr}(A^T B).$$
The corresponding norm, often referred to as the Frobenius norm, is
the ordinary $2$-norm when matrices are identified with vectors in
$\mathbb{R}^{d^2}$. For the linear ODE example, $\zeta(B) = e^{t B} x$,
and 
$$||y -  \zeta(B)||_2^2 = \mathrm{tr}(y y^T) - 2 \mathrm{tr}(e^{t B} x y^T)
- \mathrm{tr} (e^{t B^T} e^{t B} x x^T),$$
which depends on the data through the three cross products $y y^T$, $x
y^T$ and $x x^T$ only. These are $d \times d$ sufficient
transformations. We also find that 
\begin{eqnarray*}
\langle r(B),  \partial_{kl} \zeta(B) \rangle & = & \mathrm{tr}( \partial_{kl}
e^{tB} x (y^T - x^T e^{t B^T})) \\ 
& = & \mathrm{tr}( \partial_{kl} e^{t B} (x y^T - x x^T e^{tB^T})) \\
& = & t L(tB, x y^T - x x^T e^{tB^T})_{l,k}
\end{eqnarray*}
by (\ref{eq:directional}). Consequently, the entire gradient of the quadratic
loss can be computed as $- 2 t L(tB, x y^T - x x^T e^{tB^T})^T$, which
amounts to computing a single directional derivative of the
exponential map. 

We also need to compute inner products of the derivatives,
$ \partial_{kl} \zeta(B)$, of $\xi$, and to this end we observe that
\begin{eqnarray*}
\langle  \partial_{kl} \zeta(B),  \partial_{hr} \zeta(B) \rangle & = &
\mathrm{tr}(x^T \left(\partial_{kl} e^{tB}\right)^T \partial_{hr} e^{tB} x) \\
& = & \mathrm{tr}(\left(\partial_{kl} e^{tB} \right)^T \partial_{hr} e^{tB} xx^T) \\
& = & t^2 L(t B^T,  L(tB, E_{hr}) xx^T)_{k,l}. 
\end{eqnarray*}
That is, an entire column (or row) of the matrix of inner products can
be computed by computing two directional derivatives of the
exponential map. 

\section{Penalized vs. constrained optimization}

As mentioned above, our algorithms solve the penalized optimization
problem for a given sequence of $\lambda$'s. A solution,
$\hat{\beta}_{\lambda}$, for a given $\lambda$ is also a solution
to the constrained optimization problem 
$$\min_{\beta \in \Theta_{s(\lambda)}} ||y - \zeta(\beta)||_2^2$$
where $s(\lambda) = \sum_{k=1}^p \omega_k |\hat{\beta}_{\lambda,k}|$ and 
$$\Theta_s = \left\{ \beta \;\middle|\; \sum_{k=1}^p \omega_k
  |\beta_k| \leq s \right\}.$$
The value of $s(\lambda)$ is decreasing in $\lambda$. Thus the
algorithm provides a sequence of solutions to the constrained problems
for increasing values of $s$. If the sequence of $\lambda$'s is fixed,
the sequence of $s$'s will, however, be random. This is a small
nuisance in the simulation study where we want to compute the degrees
of freedom repeatedly for a fixed $s$. In practice we have solved this
by linear interpolation to compute $\widehat{\mathrm{Risk}}(s)$ for a fixed set 
of constraints $s$.  

\section{Further details and results from the simulation study}

In the simulation study on estimation of linear ODE models, data were
generated using the following sparse $10 \times 10$ matrix:
{\small
$$B = \left(
\begin{array}{rrrrrrrrrr} 
 -1.0 & -1.0 & -0.9 & -0.8 & -0.7 & -0.6 & -0.4 & -0.3 & -0.2 & -0.1 \\ 
 1.0 & -1.0 & . & . & . & . & . & . & . & . \\ 
 0.9 & . & -1.0 & . & . & . & . & . & . & . \\ 
 0.8 & . & . & -1.0 & . & . & . & . & . & . \\ 
0.7 & . & . & . & -1.0 & . & . & . & . & . \\ 
 0.6 & . & . & . & . & -1.0 & . & . & . & . \\ 
 0.4 & . & . & . & . & . & -1.0 & . & . & . \\ 
 0.3 & . & . & . & . & . & . & -1.0 & . & . \\ 
 0.2 & . & . & . & . & . & . & . & -1.0 & . \\ 
 0.1 & . & . & . & . & . & . & . & . & -1.0 \\ 
\end{array}
\right)$$
}

The matrix exponential of $B$ is a dense matrix with most of the
entries of comparable size. 

{\tiny
$$e^B = \left(
\begin{array}{rrrrrrrrrr}
 -0.11 & -0.19 & -0.17 & -0.15 & -0.12 & -0.10 & -0.08 & -0.06 & -0.04 & -0.02 \\ 
 0.19 & 0.23 & -0.12 & -0.11 & -0.09 & -0.08 & -0.06 & -0.04 & -0.03 & -0.01 \\ 
 0.17 & -0.12 & 0.26 & -0.09 & -0.08 & -0.07 & -0.05 & -0.04 & -0.03 & -0.01 \\ 
 0.15 & -0.11 & -0.09 & 0.29 & -0.07 & -0.06 & -0.05 & -0.03 & -0.02 & -0.01 \\ 
0.12 & -0.09 & -0.08 & -0.07 & 0.31 & -0.05 & -0.04 & -0.03 & -0.02 & -0.01 \\ 
 0.10 & -0.08 & -0.07 & -0.06 & -0.05 & 0.33 & -0.03 & -0.02 & -0.02 & -0.01 \\ 
 0.08 & -0.06 & -0.05 & -0.05 & -0.04 & -0.03 & 0.34 & -0.02 & -0.01 & -0.01 \\ 
 0.06 & -0.04 & -0.04 & -0.03 & -0.03 & -0.02 & -0.02 & 0.35 & -0.01 & -0.00 \\ 
 0.04 & -0.03 & -0.03 & -0.02 & -0.02 & -0.02 & -0.01 & -0.01 & 0.36 & -0.00 \\ 
 0.02 & -0.01 & -0.01 & -0.01 & -0.01 & -0.01 & -0.01 & -0.00 &
  -0.00 & 0.37 \\ 
\end{array} \right)$$
}

\begin{figure}[t]
\begin{center}
\includegraphics[width=0.6\textwidth]{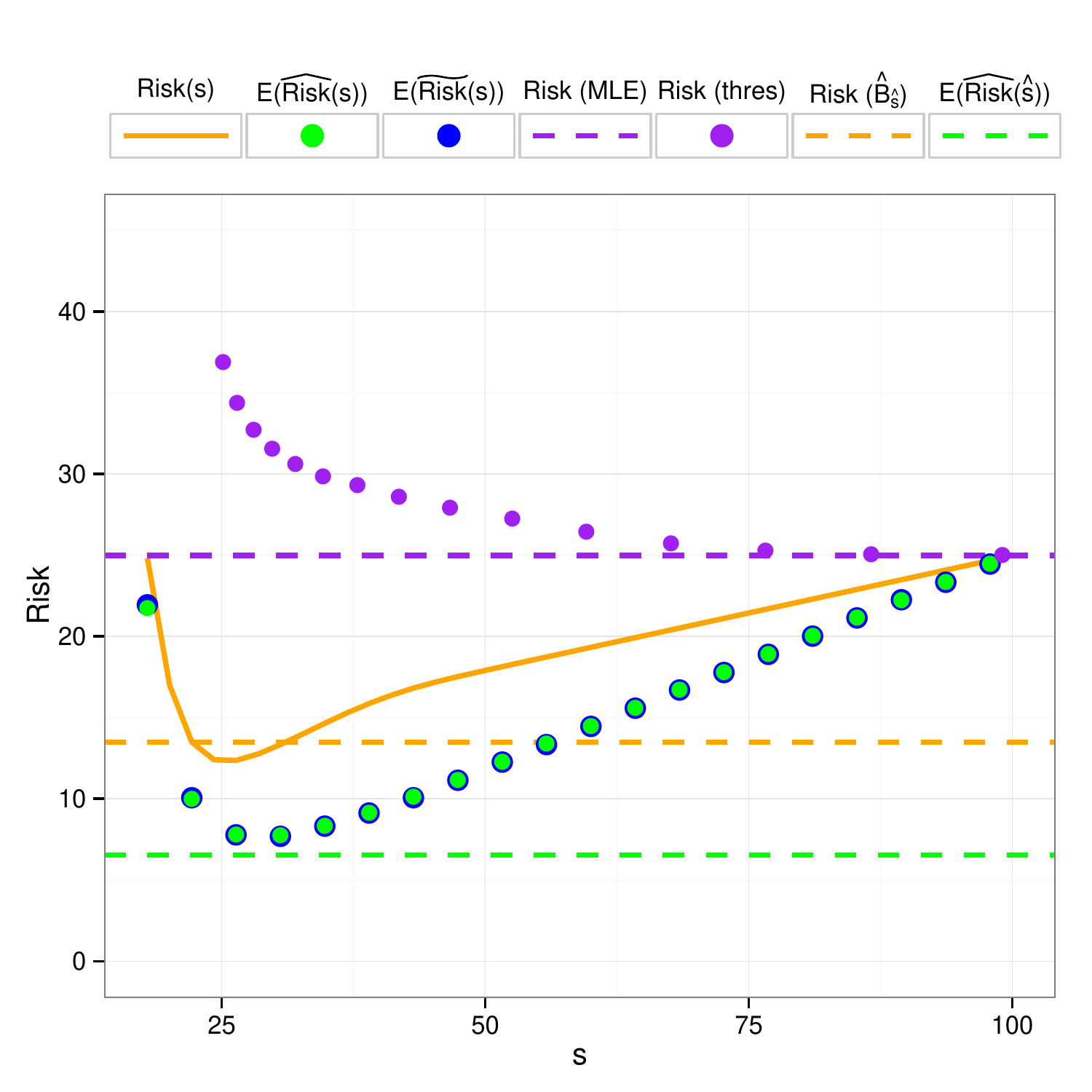} 
\end{center}
\caption{Risks for the $\ell_1$-constrained estimator with adaptive weights
  as a function of the constraint $s$ compared to the risk of
  the MLE and hard thresholding of the MLE. In addition, expected
  values of risk estimates. The 
risk estimates underestimated the true risk when adaptive weights were
used for the $\ell_1$-constrained estimator. \label{fig:riskAdap}}
\end{figure}

In addition to the results reported in the paper, 
Figure \ref{fig:riskAdap} shows the results for the $\ell_1$-constrained
estimator with adaptive weights. Using the divergence as an estimate
of degrees of freedom resulted in this case in negatively biased 
risk estimates. This is because the divergence does not account
for the data dependent weights. Despite of this, the data adaptive
choice of the constraint was close to the optimal choice. It is
notable that compared to using unit weights, the use of adaptive
weights decreased the risk further. The adaptive weights also resulted in
sparser estimates (41.7 nonzero entries on average) than when using unit weights (59.0
nonzero entries on average). Using $\hat{B}_{\hat{s}}$ to obtain 
a structural estimate of the nonzero entries the accuracy (fraction
of correctly estimated zero and nonzero entries) was
0.81 with adaptive weights compared to 0.65 with unit
weights. The forward stepwise model search gave 37.3 nonzero entries
on average and an accuracy of 0.83. Thus for structural estimation, the
model search was more accurate, though this comparison may not be entirely
fair. The model search started from the diagonal matrix mainly 
for numerical reasons, which gave it 10 correct nonzero entries as a
starting point. 

\begin{figure}
\begin{center}
  \includegraphics[width=0.9\textwidth]{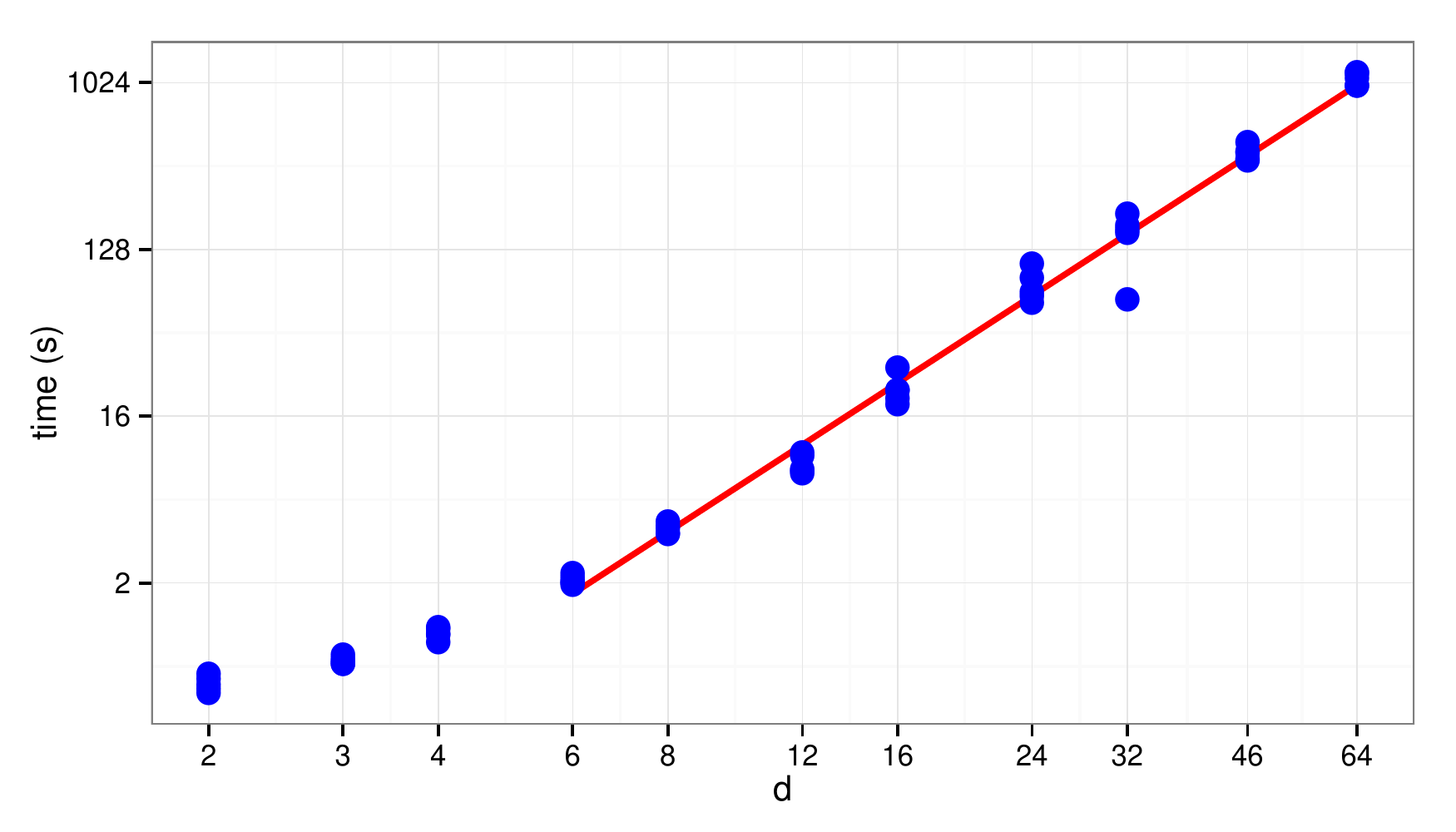} 
\end{center}
\caption{Computation time in seconds for Algorithm II as a function of the dimension
  $d$. The read line has slope 2.7. \label{fig:scaling}}
\end{figure}

We finally carried out a small benchmark simulation to investigate how
the current implementation of the coordinate descent algorithm for the
nonlinear least squares problem scales with the dimension of the
problem. It should be noted that there are many nobs to tweak to
improve computation times. The simulation results presented in the paper
with $d = 10$ were, for instance, carried out with a small relative tolerance of
around $10^{-8}$ for the convergence criterion. In this benchmark
study we used a relative tolerance of $10^{-4}$, which in our
experience only occasionally will result in convergence problems. It is also possible
to stop the algorithm when the model with the minimal estimated risk
is reached to avoid the most expensive part of the optimization where
many parameters are nonzero. We have not done that, but as
in the simulation study in the paper we computed the optimal solution for
40 precomputed values of the penalty parameter. Then there is the
specific choice of algorithm. In the benchmark we used Algorithm
II described above. 

Figure \ref{fig:scaling} shows the computation times for the
optimization as a function of the dimension $d$ 
for $d \in \{2, 3, 4, 6, 8, 12, 16, 24, 32, 46, 64\}$. Note that the
number of parameters is $p = d^2$, which for $d = 64$ gives $p = 4096$ parameters. For each value
of $d$ we made 5 replications. We see from Figure \ref{fig:scaling}
that the computation time scales roughly
like $d^3$. The bottleneck is the repeated computations of dense
matrix exponentials.

\bibliographystyle{agsm}
\bibliography{../texbib-1}